\documentclass[notitlepage]{amsart}

\usepackage[all]{xy}
\usepackage{pb-diagram,pb-xy}

\usepackage{amsmath,amscd}
\usepackage{enumerate}
\usepackage{amsmath}
\usepackage{amssymb}
\usepackage{amsfonts}
\usepackage{eufrak}
\usepackage{latexsym}
\usepackage{leftidx}
\usepackage{amsthm}
\usepackage{color}

\newtheorem{theorem}{Theorem}[section]
\newtheorem{proposition}[theorem]{Proposition}
\newtheorem{lemma}[theorem]{Lemma}
\newtheorem{defin}[theorem]{Definition}
\newtheorem{corollary}[theorem]{Corollary}
\newtheorem{remark}[theorem]{Remark}
\newtheorem{example}[theorem]{Example}

\newcommand{\modu}{{\mathrm{mod}}}
\newcommand{\Ima}{{\mathrm{Im}}}
\newcommand{\Ker}{{\mathrm{Ker}}}

\newcommand{\Hom}{{\mathrm{Hom}}}
\newcommand{\rk}{{\mathrm{rk}}}
\newcommand{\Tr}{{\mathrm{Tr}}}
\newcommand{\End}{{\mathrm{End}}}
\newcommand{\Ext}{{\mathrm{Ext}}}
\newcommand{\Tor}{{\mathrm{Tor}}}
\newcommand{\proj}{{\mathrm{proj}}}

\newcommand{\N}{{\mathbb{N}}}

\newcommand{\F}{{\mathcal{F}}}
\newcommand{\Q}{{\mathbf{Q}}}

\newcommand{\C}{{\mathcal{C}}}

\newcommand{\X}{{\mathcal{X}}}
\newcommand{\Y}{{\mathcal{Y}}}
\newcommand{\add}{{\mathrm{add}}}
\newcommand{\rad}{{\mathrm{rad}}}
\newcommand{\Top}{{\mathrm{top}}}

\renewenvironment{proof}{\noindent \bf Proof. \rm}{$\quad \hfill \square$ \bigskip}
\begin{document}


\title{$\mathcal{C}$-filtered modules and proper costratifying systems}

\author{\sc O. Mendoza, M.I. Platzeck and M. Verdecchia}

%\date{ }

\maketitle

\begin{abstract}
In this paper we define and study the notion of a proper costratifying
system, which is a generalization of the so-called proper costandard modules to the
context of stratifying systems. The proper costandard modules were
defined by V. Dlab in his study of quasi-hereditary algebras (see \cite{Dlab}).
\end{abstract}

\section*{Introduction.}
Standard modules over artin algebras were defined by C.M. Ringel (see \cite{R}) in connection with the study of quasi-hereditary algebras,
 where the category of modules filtered by them plays an essential role. Let $P(1), ..., P(n)$ be an ordered sequence of the
non-isomorphic indecomposable projective modules over an artin algebra $\Lambda$. By definition, the standard module ${}_\Lambda\Delta(i)$ is
the largest factor module of $P(i)$ with composition factors only amongst $S(1), ..., S(i)$, where $S(j)$ is the simple top of $P(j)$.
Let mod$(\Lambda)$ denote the category of finitely generated left $\Lambda$-modules. Denote by $\mathcal{F}({}_\Lambda\Delta)$ the subcategory
of mod$(\Lambda)$ consisting of the $\Lambda$-modules having a filtration with factors isomorphic to standard modules. The algebra
$\Lambda$ is said to be standardly stratified if all projective $\Lambda$-modules belong to $\mathcal{F}({}_\Lambda\Delta)$. This class
of algebras was originally defined by Cline, Parshall and Scott in \cite{CPS}, and was widely studied by different mathematicians
(see \cite{ADL}, \cite{AHLU}, \cite{Webb}, \cite{ES}, \cite{Platzeck Reiten},  \cite{Xi}).
 \

Erdmann and S\'aenz extended the notion of standard modules and defined the stratifying systems in \cite{ES} with respect to a finite
linear ordered set. They proved that, for a
stratifying system $\Theta$, the category of modules filtered by $\Theta$ is equivalent to the category of modules filtered by the standard
modules over an appropriate standardly stratified algebra. The same was done in \cite{MSXi} for stratifying systems defined on a finite
pre-ordered set.  Important work in this direction was done in the earlier paper \cite{Webb} of Webb, though stratifying systems were not
defined there.
 \

In contrast with the situation for quasi-hereditary algebras, if $\Lambda$ is a standardly stratified algebra then $\Lambda^{op}$ need not also be standardly stratified. However, Dlab defined a new class of modules, the proper standard modules (see \cite{Dlab}), with the property that $\Lambda$ is a standardly stratified algebra, that is, $\Lambda$ is filtered by the standard modules, if and only if $\Lambda^{op}$ is filtered by the proper standard modules. This motivates the study of the category of modules filtered by the proper standard modules (see \cite{AHLU}, \cite{L}).
  \

In this paper we define and study the notion of a proper costratifying
system, which is a generalization of the so-called proper costandard modules to the
context of stratifying systems.
  \

One of our main results states that the category of modules filtered by a proper costratifying system is dual to the category of modules filtered by the proper costandard modules over a certain standardly stratified algebra.
  \

Although stratifying systems and proper costratifying systems are quite different, they have similar features, and they can be studied under
a common frame. This comes from the observation that, in either case, there is a module $M$ in mod$(\Lambda)$ such that the category $\mathcal{F}$ of modules filtered by the corresponding system satisfies the following property:
\begin{quotation}
For each $X$ in $\mathcal{F}$ there is an exact sequence $M_2\rightarrow M_1\rightarrow M_0\rightarrow X\rightarrow0$ with $M_0,M_1,M_2$ in add$(M)$,
which remains exact under the functor $F=\text{Hom}_\Lambda(M,-)$. Here $\add\,(M)$ denotes the full subcategory of $\modu\,(\Lambda)$ consisting of the direct sums of direct summands of $M$,
\end{quotation}
as proved in Lemma \ref{lema MMS sist. estratif theta} and Proposition \ref{dim(Hom(Yi,titai))=1; sucesion exacta M', E y M}.

   \
For a $\Lambda$-module $M$, the category consisting of the $\Lambda$-modules admitting a sequence as above, is denoted by $C^{M}_2$,
and was studied by Platzeck and Pratti in \cite{Platzeck Pratti}. In both cases $M$ is Ext-projective in $\mathcal{F}$. More precisely,
$M$ is the sum of the non-isomorphic Ext-projective indecomposable modules in $\mathcal{F}$. Using this fact, results in \cite{Platzeck Pratti} can be applied to prove properties of the category $\mathcal{F}$.
    \

The paper is organized as follows. After a brief section of preliminaries, we devote section 2 to the study of $C^{M}_2$ categories and show how these results apply to give a new proof of the above mentioned theorem of Erdmann and S\'aenz concerning stratifying systems. In section 3 we introduce the notion of a stratifying system and study their properties. Finally, in section 4 we prove our main results about proper costratifying systems.

\section{Preliminaries}
\

Throughout this paper algebra means \emph{artin $R$-algebra}, where $R$ is a commutative artinian ring. When
$\Lambda$ is an algebra the term \lq $\Lambda$-module\rq   will mean \emph{finitely
generated left $\Lambda$-module}. The category of finitely
generated left $\Lambda$-modules is denoted by $\modu\,(\Lambda)$ and the full
subcategory of finitely generated projective $\Lambda$-modules by
$\proj\,(\Lambda)$. For $\Lambda$-modules $M$ and $N$,
$\Tr_M\,(N)$ is the \emph{trace} of $M$ in $N$, that is, $\Tr_M\,(N)$ is
the $\Lambda$-submodule of $N$ generated by the images of all
morphisms from $M$ to $N$. Let $D:\modu\,(\Lambda)\rightarrow\modu\,(\Lambda^{op})$ denote the \emph{usual duality} for artin algebras, and $*$ denote the functor $\Hom_\Lambda(-,\Lambda):\modu\,(\Lambda)\rightarrow\modu\,(\Lambda^{op})$. Then $*$ induces a duality from $\proj\,(\Lambda)$ to
$\proj\,(\Lambda^{op}).$ For a given natural number $t,$ we set
$[1,t]=\{1,2,\cdots, t\}.$
\

Let $\Lambda$ be an algebra. We next recall the definition (see \cite{R,DR,ADL,Dlab}) of the following classes of $\Lambda$-modules: standard, proper standard, costandard and proper costandard. Let $n$ be the rank of the Grothendieck group $K_0\,(\Lambda)$. We fix a linear order $\leq$ on $[1,n]$ and a representative set ${}_\Lambda P=\{{}_{\Lambda}P(i)\;:\;i\in[1,n]\}$ containing one module of each iso-class of indecomposable projective $\Lambda$-modules. The injective envelope of the simple $\Lambda$-module ${}_{\Lambda}S(i)=\Top\,({}_{\Lambda}P(i))$ is denoted by ${}_{\Lambda}I(i).$ For the opposite algebra $\Lambda^{op},$ we always consider the representative set ${}_{\Lambda^{op}} P=\{{}_{\Lambda^{op}}P(i):i\in[1,n]\}$ of indecomposable projective $\Lambda^{op}$-modules, where ${}_{\Lambda^{op}}P(i)=({}_{\Lambda}P(i))^*$ for all $i\in[1,n]$. So, with these choices in mind, we introduce now the following classes of modules:
\

The set of \emph{standard $\Lambda$-modules} is $\leftidx{_\Lambda}{\Delta}=\{
\leftidx{_\Lambda}{\Delta}(i):i\in[1,n]\},$ where
$\leftidx{_\Lambda}{\Delta}(i)={}_{\Lambda}P(i)/\Tr_{\oplus_{j>i}\,{}_{\Lambda}P(j)}\,({}_{\Lambda}P(i))$. Then,
${}_{\Lambda}{\Delta}(i)$ is the largest factor module of
${}_{\Lambda}P(i)$ with composition factors only amongst
${}_{\Lambda}S(j)$ for $j\leq i.$ The set of \emph{costandard
$\Lambda$-modules} is
${}_{\Lambda}\!\nabla=D({}_{\Lambda^{op}}{\Delta}),$
where the pair $({}_{\Lambda^{op}}P,\leq)$ is used to compute
${}_{\Lambda^{op}}\Delta.$
 \

The set of  \emph{proper standard $\Lambda$-modules} is
${}_{\Lambda}\overline{\Delta}=\{{}_{\Lambda}\overline{\Delta}(i):i\in[1,n]\},$
where
${}_{\Lambda}\overline{\Delta}(i)={}_{\Lambda}P(i)/\Tr_{\oplus_{j\geq
i}\,{}_{\Lambda}P(j)}\,(\rad\,{}_{\Lambda}P(i))$. Then,
${}_{\Lambda}\overline{\Delta}(i)$ is the largest factor module of
${}_{\Lambda}{\Delta}(i)$ satisfying the multiplicity condition
$[{}_{\Lambda}\overline{\Delta}(i):S(i)]=1.$ The set of \emph{proper
costandard $\Lambda$-modules} is
${}_{\Lambda}\!\overline{\nabla}=D({}_{\Lambda^{op}}\overline{\Delta}),$
where the pair $({}_{\Lambda^{op}}P,\leq)$ is used to compute
${}_{\Lambda^{op}}\overline{\Delta}.$
 \

Let $\F({}_{\Lambda}{\Delta})$ be the subcategory of $\modu\,(\Lambda)$ consisting of the
$\Lambda$-modules having a ${}_{\Lambda}{\Delta}$-filtration, that is,
a filtration  $0=M_0\subseteq M_1\subseteq\cdots\subseteq M_s=M$ with factors $M_{i+1}/M_i$
isomorphic to a module in ${}_{\Lambda}{\Delta}$ for all $i$. The algebra $\Lambda$ is
a \emph{standardly stratified algebra} with respect to the linear
order $\leq$ on the set $[1,n]$, if
$\proj\,(\Lambda)\subseteq\F({}_{\Lambda}{\Delta})$ (see \cite{ADL,Dlab,CPS}).
The algebra $\Lambda$ is
a \emph{properly stratified algebra} with respect to the linear
order $\leq$ on the set $[1,n]$, if and only if its regular representation is filtered by standard as well as by proper standard modules. That is,
$\proj\,(\Lambda)\subseteq\F({}_{\Lambda}{\Delta})\cap\F({}_{\Lambda}{\overline{\Delta}})$ (see \cite{Dlab2}).
 \

Recall that a morphism $f:C\rightarrow M$ in $\modu\,(\Lambda)$ is
\emph{right minimal} if any morphism $g:C\rightarrow C,$ with
$f=fg,$ is an automorphism. Moreover, for a given class $\C$ of objects in
$\modu\,(\Lambda)$, $f:C\rightarrow M$ is a
\emph{right $\mathcal{C}$-approximation} of $M$ if $C\in\mathcal{C}$
and the map
$\Hom_\Lambda(C_1,f):\Hom_\Lambda(C_1,C)\rightarrow\Hom_\Lambda(C_1,M)$ is surjective for all $C_1\in\C$. A \emph{right minimal
$\mathcal{C}$-approximation} is a right $\mathcal{C}$-approximation
which is right minimal. The notion of
\emph{left minimal} morphism and \emph{left minimal
$\mathcal{C}$-approximation} are defined dually.
\

Let $\Lambda$ be an algebra and $\X$ a class of objects in $\modu\,(\Lambda).$ For
each natural number $n,$ we set ${}^{\perp_n}\X=\{M\in\modu\,(\Lambda): \Ext_\Lambda^n(M,-)|_{\X}=0\}$
and ${}^{\perp}\X=\cap_{n>0}\,{}^{\perp_n}\X.$ Similarly, the notions of $\X^{\perp_n}$
and $\X^{\perp}$ are introduced.

\section{$C^{M}_2$ categories and $\mathcal{C}$-filtered modules.}
The categories $C^{M}_n$, whose definition is recalled in the next paragraph, were introduced by Platzeck and Pratti in \cite{Platzeck Pratti}, where particular interest was focused on the
 case when $C^{M}_0=C^{M}_1$. Here, we will apply these ideas in a different context. We will mainly concentrate in the case when
$\mathcal{C}\subseteq C^{M}_2$, and study properties of the category $\mathcal{F}(\mathcal{C})$ of modules filtered by $\mathcal{C}$.
These results apply to the category of modules filtered by a stratifying system (Theorem \ref{teo. sist. estratif. Ext-proyectivos MMS}),
as well as to those filtered by a proper costratifying system (Theorem \ref{Teo. equivalencia}) , since both categories are contained
 in $C^{Q}_2$, for an appropriate $Q$. Other examples of categories contained in $C^{M}_2$ are the torsion modules of a tilting module $M$.
This gives a new approach to prove well-known results in tilting theory (\cite{Platzeck Pratti}, \cite{Platzeck Pratti2}).
\

Let $\Lambda$ be an artin $R$-algebra. For each
$M\in\modu\,(\Lambda)$, we consider the opposite algebra
$\Gamma=\End_\Lambda({M})^{op}$ and the $R$-functors
 $$\modu\,(\Lambda)\overset{F}{\longrightarrow}\modu\,(\Gamma)\overset{G}{\longrightarrow}\modu\,(\Lambda),$$
where $F=\Hom_\Lambda({}_{\Lambda}M_\Gamma,-)$ and $G={}_\Lambda
M_\Gamma\otimes_\Gamma-$. Following M.I. Platzeck and N.I. Pratti in Section 2 of
\cite{Platzeck Pratti}, for $n\geq 0$ we denote by $C^{M}_n$ the
full subcategory of $\modu\,(\Lambda)$ consisting of the
$\Lambda$-modules $X$ admitting  an exact sequence in
$\modu\,(\Lambda)$
$$M_n\rightarrow M_{n-1}\rightarrow...\rightarrow M_1\rightarrow
M_0\rightarrow X\rightarrow 0$$
 with $M_i\in\add\,(M)$, and such that the induced sequence
 $$F(M_n)\rightarrow F(M_{n-1})\rightarrow ...\rightarrow F(M_ 1)\rightarrow F(M_0)\rightarrow F(X)\rightarrow 0$$
is exact in $\modu\,(\Gamma)$.
\

We recall next the following useful result due to M. Auslander in \cite{Auslander}.

\begin{theorem} \label{teorema de Auslander}
 For $M\in \modu\,(\Lambda)$, $\Gamma=\End_\Lambda({M})^{op}$ and $F=\Hom_\Lambda(M,-):\modu\,(\Lambda)\rightarrow\modu\,
 (\Gamma)$, the following statements hold.
\begin{itemize}
\item [(a)] The restriction $F|_{C^{M}_1}:C^{M}_1 \rightarrow \modu\,(\Gamma)$ is full and faithful.
\item [(b)] $F:\Hom_\Lambda(Z,X)\rightarrow\Hom_\Gamma(F(Z),F(X))$ is an isomorphism in
$\modu\,(R)$, for all $Z\in\add\,(M)$, $X\in\modu\,(\Lambda).$
\item [(c)] The restriction $F|_{\add\,(M)}:\add\,(M) \rightarrow \proj\,(\Gamma)$ is an equivalence of $R$-categories.
\end{itemize}
\end{theorem}

\begin{remark}
\label{observacion 2}
\rm As a consequence of Theorem \ref{teorema de Auslander} (b) and the left exactness of $F$, it follows that if $X\in C^{M}_2 $ then there exists an exact sequence in $\modu\,(\Lambda)$,
$0\rightarrow K\rightarrow M_0\rightarrow X\rightarrow 0,$
with $M_0\in\add\,(M)$ and $K\in C^{M}_1$, such that the sequence
$0\rightarrow F(K)\rightarrow F(M_0)\rightarrow F(X)\rightarrow 0$ is exact in $\modu\,(\Gamma).$
\end{remark}

The following result, proven in
\cite{Platzeck Pratti}, will be very useful in what follows, where
$\epsilon:GF\to 1$ is the co-unit of the adjunction
$\eta:\Hom_\Lambda(G-,-)\to\Hom_\Gamma(-,F-)$, that is,
$\epsilon_X=\eta^{-1}(1_{F(X)}):GF(X)\to X.$
\begin{proposition} \label{prop. C1M-isom.epsilon}  Let $M\in \modu\,(\Lambda)$, $\Gamma=\End_\Lambda(M)^{op}$,
$F=\Hom_\Lambda(M,-):\modu\,(\Lambda) \rightarrow
\modu\,(\Gamma)$ and
$G=M\otimes_\Gamma-:\modu\,(\Gamma)\rightarrow\modu\,(\Lambda).$
 Then $$C^{M}_1\subseteq\{X\in\modu\,(\Lambda)\text{ such that }\epsilon_X:GF(X)\rightarrow X\text{ is an isomorphism }\}.$$
\end{proposition}
\begin{proof} See \cite[Proposition 2.2]{Platzeck Pratti}.
\end{proof}

The next propositions show that the modules in $C^{M}_2$ have nice homological properties.

\begin{proposition} \label{prop. iso entre Ext}
 Let $M\in\modu\,(\Lambda)$, $\Gamma=\End_\Lambda(M)^{op},$ $F=\Hom_\Lambda(M,-):\modu\,(\Lambda) \rightarrow \modu\,(\Gamma)$
 and $\Y\subseteq C^{M}_1$ be such that $M\in\leftidx{^{\bot_1}}{\mathcal{Y}}$. Then, for all $X\in C^{M}_2$,
$Y\in \mathcal{Y}$, the map induced by $F$ $$\rho_{X,Y}:\text{\rm Ext}_\Lambda ^{1}(X,Y)\rightarrow\text{\rm Ext}_\Gamma ^{1}(F(X),F(Y))$$
is an isomorphism of $R$-modules.
\end{proposition}
\begin{proof}
Let $X\in C^{M}_2 $ and $Y\in \mathcal{Y}$. By Remark \ref{observacion 2} there exists an exact sequence
$$\varepsilon:0\rightarrow K\rightarrow M_0\rightarrow X\rightarrow 0,$$
with $K\in C^{M}_1$ and $M_0\in\text{\rm add}\,(M)$, such that the sequence
$$F(\varepsilon):0\rightarrow F(K)\rightarrow F(M_0)\rightarrow F(X)\rightarrow 0$$
is exact in $\text{\rm mod}\,(\Gamma)$. Since $M_0\in\add\,(M)$ and $M\in{}^{\perp_1}\Y,$ we have that
$\text{\rm Ext}_\Lambda ^{1}(M_0,Y)=0=\text{\rm Ext}_\Gamma ^{1}(F(M_0),F(Y))$ because $F(M_0)\in\proj\,(\Gamma)$
(see Theorem \ref{teorema de Auslander} (c)). Then, by applying $\Hom_\Lambda(-,Y)$ to $\varepsilon$, and $\Hom_\Gamma(-,F(Y))$ to
$F(\varepsilon)$, we get the following exact and commutative diagram
$$ \begin {diagram} \dgARROWLENGTH=1em \scriptsize
\node {\text{\rm Hom}_\Lambda(M_0,Y)} \arrow {e,l}{} \arrow {s,r}{\simeq}
\node {\text{\rm Hom}_\Lambda(K,Y)} \arrow {e,l}{} \arrow {s,r}{\simeq}
\node {\text{\rm Ext}_\Lambda ^{1}(X,Y)} \arrow {e,l}{} \arrow {s,r}{\rho_{X,Y}}
\node {0} \\
\node {\text{\rm Hom}_\Gamma(F(M_0),F(Y))} \arrow {e,l}{}
\node {\text{\rm Hom}_\Gamma(F(K),F(Y))} \arrow {e,l}{}
\node {\text{\rm Ext}_\Gamma ^{1}(F(X),F(Y))} \arrow {e,l}{}
\node {0.}
\end {diagram}$$
By Theorem \ref{teorema de Auslander}, we have that the two first vertical arrows are isomorphisms.
Hence $\rho_{X,Y}$ is an isomorphism.
\end{proof}
\begin{proposition} \label{lema Tor y lema2}
Let $M\in\modu\,(\Lambda)$, $\Gamma=\End_\Lambda(M)^{op}$ and
$F=\Hom_\Lambda(M,-):\modu\,(\Lambda) \rightarrow \modu\,(\Gamma)$.
Then the following statements hold.
\begin{itemize}
\item[(a)] $F(C^{M}_2)\subseteq\Ker\,\Tor^{\Gamma}_1(M,-)$.
\item[(b)] If $\varepsilon:0\rightarrow F(X)\rightarrow Y'\rightarrow F(Z)\rightarrow 0$ is exact in $\modu\,(\Gamma)$, with
$X,Z\in C^{M}_2$, then there is an exact sequence $\eta:0\rightarrow X\rightarrow Y\rightarrow Z\rightarrow 0$ in $\modu\,(\Lambda)$
such that $\varepsilon\simeq F(\eta)$.
\end{itemize}
\end{proposition}
\begin{proof}
We consider the functor $G=M\otimes_\Gamma-:\modu\,(\Gamma)\rightarrow\modu\,(\Lambda).$\\
\indent (a) The arguments in the proof of \cite[Theorem 3.7]{Platzeck Pratti} can be easily adapted to this case.
% Let $X\in C^{M}_2 $. By Remark \ref{observacion 2} there is an exact sequence in $\modu\,(\Lambda)$
%$$\mu:0\rightarrow K\rightarrow M_0\rightarrow X\rightarrow 0,$$
%with $K\in C^{M}_1$ and $M_0\in\text{\rm add}\,(M)$, such that the sequence
%$$F(\mu):0\rightarrow F(K)\rightarrow F(M_0)\rightarrow F(X)\rightarrow 0$$
%is exact in $\modu\,(\Gamma)$. Applying the functor $G$ to $F(\mu)$, and since $F(M_0)\in\proj\,(\Gamma)$, we get
%the exact and commutative diagram
%$$ \begin {diagram} \dgARROWLENGTH=0.7em \small
%\node {0} \arrow {e,l}{}
%\node {\text{\rm Tor}^{\Gamma}_{1}(M,F(X))} \arrow {e,l}{}
%\node {GF(K)} \arrow {e,l}{} \arrow {s,r}{\epsilon_K}
%\node {GF(M_0)} \arrow {e,l}{} \arrow {s,r}{\epsilon_{M_0}}
%\node {GF(X)} \arrow {e,l}{} \arrow {s,r}{\epsilon_X}
%\node {0} \\
%\node {}
%\node {0} \arrow {e,l}{}
%\node {K} \arrow {e,l}{}
%\node {M_0} \arrow {e,l}{}
%\node {X} \arrow {e,l}{}
%\node {0}
%\end {diagram}$$
%where \textcolor{red}{the vertical arrows are isomorphisms}
%because $K,M_0,X\in C^{M}_1$ (see \ref{prop. C1M-isom.epsilon}).
%Therefore, $\Tor^{\Gamma}_{1}(M,F(X))=0$. \

(b) Let $\varepsilon:0\rightarrow F(X)\rightarrow
Y'\stackrel{g}{\rightarrow} F(Z)\rightarrow 0$ be exact in
$\modu\,(\Gamma)$,
 with $X,Z\in C^{M}_2$. Applying $G$ to $\varepsilon$, we have the exact sequence
$$\text{\rm Tor}^{\Gamma}_{1}(M,F(Z))\rightarrow GF(X)\rightarrow G(Y')\rightarrow GF(Z)\rightarrow 0.$$
Since $C^{M}_2\subseteq C^{M}_1$, from the last sequence and (a), we get the exact and commutative diagram
$$ \begin {diagram} \dgARROWLENGTH=1em
\node {G(\varepsilon):0} \arrow {e,l}{}
\node {GF(X)} \arrow {e,l}{} \arrow {s,r}{\simeq}
\node {G(Y')} \arrow {e,l}{} \arrow {s,r,=}{}
\node {GF(Z)} \arrow {e,l}{} \arrow {s,r}{\simeq}
\node {0} \\
\node {\;\eta:0} \arrow {e,l}{}
\node {X} \arrow {e,l}{}
\node {G(Y')} \arrow {e,l}{}
\node {Z} \arrow {e,l}{}
\node {0.}
\end{diagram}$$
We  will show next that $\eta$ is the desired sequence, where $Y=G(Y')$. Indeed, by applying the left exact functor
$F$ to $G(\varepsilon)$, we obtain the following exact and commutative diagram
$$ \begin {diagram} \dgARROWLENGTH=1em
\node {0} \arrow {e,l}{}
\node {F(X)} \arrow {e,l}{} \arrow {s,r}{\simeq}
\node {Y'} \arrow {e,l}{g} \arrow {s,r}{h}
\node {F(Z)} \arrow {e,l}{} \arrow {s,lr}{\theta}{\simeq}
\node {0} \\
\node {0} \arrow {e,l}{}
\node {FGF(X)} \arrow {e,l}{}
\node {FG(Y')} \arrow {e,l}{FG(g)}
\node {FGF(Z)} %\arrow {e,l}{}
%\node {0.}
\end{diagram}$$
From the equality $\theta g=FG(g)h$ and the fact that $\theta$ is an isomorphism, it follows that $FG(g)$
is an epimorphism. Hence $h$ is an isomorphism, and this proves that $\varepsilon\simeq F(\eta)$.
\end{proof}

For a functor $F:\mathcal{A}\rightarrow\mathcal{B}$ and a class of objects $\X$ in $\mathcal{A}$, let
$F(\X)=\{Z\in\mathcal{B}:Z\simeq F(X) \text{ for some
}X\in\X\}$.

\begin{corollary} \label{Corol. F(X) cerrada por extensiones}
 Let $M\in\modu\,(\Lambda)$, $\Gamma=\End_\Lambda(M)^{op},$
 $F=\Hom_\Lambda(M,-):\modu\,(\Lambda) \rightarrow \modu\,(\Gamma)$
 and $\X\subseteq C^{M}_2.$ If $\X$ is closed under extensions, then
 $F(\X)$ is so.
\end{corollary}
\begin{proof}
The proof follows immediately from Proposition \ref{lema Tor y lema2} (b).
\end{proof}

Let $\Lambda$ be an algebra and $\C$ be a class of objects in
$\modu\,(\Lambda)$. We denote by $\F(\C)$ the class of the
$\Lambda$-modules having a $\C$-filtration, that is, a filtration $0=M_0\subseteq
M_1\subseteq\cdots\subseteq M_n=M$ of submodules with factors
$M_{i+1}/M_i$ isomorphic to a module in $\C$ for all $i$. Then $\mathcal{F}(\mathcal{C})$ is the smallest class in
$\modu\,(\Lambda)$ which is closed under extensions and contains
$\C$. Moreover, it is straightforward to see that
${}^{\perp_1}\C={}^{\perp_1}\F(\C).$

\begin{corollary} \label{Corol. F(C)=mathcal{F}(F(C))}
 Let $M\in\modu\,(\Lambda)$, $\Gamma=\End_\Lambda(M)^{op},$
 $F=\Hom_\Lambda(M,-):\modu\,(\Lambda) \rightarrow \modu\,(\Gamma)$
 and $\C\subseteq\modu\,(\Lambda)$. If the restriction $F|_{\F(\C)}:\F(\C)\to \modu\,(\Gamma)$
is an exact functor and $\F(\C)\subseteq C^{M}_2$, then
$F(\mathcal{F}(\mathcal{C}))=\mathcal{F}(F(\mathcal{C})).$
\end{corollary}
\begin{proof} Since the restriction $F|_{\F(\C)}:\F(\C)\to\modu\,(\Gamma)$ is an exact functor, it follows that
$F(\mathcal{F}(\mathcal{C}))\subseteq\mathcal{F}(F(\mathcal{C})).$
On the other hand, the condition $\mathcal{F}(\mathcal{C})\subseteq
C^{M}_2$ and Corollary \ref{Corol. F(X) cerrada por extensiones} give us the
other inclusion.
% $\mathcal{F}(F(\mathcal{C}))\subseteq
%F(\mathcal{F}(\mathcal{C}))$ since $F(\C)\subseteq F(\F(\C))$.
\end{proof}

We recall that a class $\X$ of objects in $\modu\,(\Lambda)$ is
\emph{resolving} if it is closed under extensions, kernels of
epimorphisms and $\proj\,(\Lambda)\subseteq\X$ (see \cite{AR}).

\begin{lemma}\label{X resolvente}
If $\X$ is a resolving subcategory of $\modu\,(\Lambda)$, then
$\proj\,(\Lambda)=\X\cap\leftidx{^{\bot_1}}{\mathcal{X}}$.
\end{lemma}
\begin{proof} Assume that $\X$ is resolving. It is clear that
$\proj\,(\Lambda)\subseteq\mathcal{X}\cap\leftidx{^{\bot_1}}{\mathcal{X}}$.
The proof is completed by showing the other inclusion. Let
$X\in\mathcal{X}\cap\leftidx{^{\bot_1}}{\mathcal{X}}$, and consider
the exact sequence $\varepsilon:0\rightarrow K\rightarrow
P_0(X)\rightarrow X\rightarrow 0$ in $\modu\,(\Lambda),$ where
$P_0(X)$ is the projective cover of $X$. Since $\mathcal{X}$ is resolving, we conclude that $K\in\mathcal{X}$,
and hence $\varepsilon$ splits, because $X\in\leftidx{^{\bot_1}}{\mathcal{X}}$. Thus $X\in\proj\,(\Lambda)$.
\end{proof}

The following lemma will be useful in the sequel.

\begin{lemma} \label{lema cuasi-inversa de F}
 Let $M\in\modu\,(\Lambda)$, $\Gamma=\End_\Lambda(M)^{op},$
 $F=\Hom_\Lambda(M,-):\modu\,(\Lambda) \rightarrow \modu\,(\Gamma)$
  and $G=M\otimes_\Gamma -:\modu\,(\Gamma)\rightarrow \modu\,(\Lambda).$
  Let $\mathcal{A}$ and $\mathcal{B}$ be full subcategories of $\modu\,(\Lambda)$ and
  $\modu\,(\Gamma)$ respectively, closed under isomorphisms and such that the restriction $F|_\mathcal{A}:\mathcal{A}\rightarrow\mathcal{B}$ is an equivalence of
  categories. If $\epsilon_A:GF(A)\rightarrow A$ is an isomorphism for all $A\in\mathcal{A}$,
  then the restriction $G|_\mathcal{B}:\mathcal{B}\rightarrow\mathcal{A}$ is a quasi-inverse of $F|_\mathcal{A}.$
\end{lemma}
\begin{proof} Let $B\in\mathcal{B}.$
First, we prove that $G(B)\in\mathcal{A}$. Indeed, since
$B\in\mathcal{B}$ and $F|_\mathcal{A}$ is dense, there exists an isomorphism
$\rho:B\rightarrow F(A)$ in $\mathcal{B}$ for some
$A\in\mathcal{A}.$ Therefore $G(B)\simeq GF(A)\simeq A$ and hence
$G(B)\in\mathcal{A}.$ \

Let now $\mu:1\to FG$ denote the
unit of the adjunction
$\eta:\Hom_\Lambda(G-,-)\to\Hom_\Gamma(-,F-)$, that is,
$\mu_Y=\eta(1_{G(Y)}):Y\to FG(Y).$ We next prove that the natural
transformation $\mu_B:B\rightarrow FG(B)$ is an isomorphism for all
$B\in\mathcal{B}$. To do so, we consider the following commutative
diagram
$$ \begin {diagram} \dgARROWLENGTH=1em
\node {B} \arrow {e,l}{\mu_B} \arrow {s,lr}{\rho}{\simeq}
\node {FG(B)} \arrow {s,lr}{FG(\rho)}{\simeq}\\
\node {F(A)} \arrow {e,r}{\mu_{F(A)}}
\node {FGF(A)} \arrow {e,r}{F(\epsilon_A)}
\node {F(A).}
\end{diagram}$$
Observe that $F(\epsilon_A)$ is an isomorphism since $\epsilon_A$ is
so. From this and the fact that $F(\epsilon_A)\mu_{F(A)}=1_{F(A)}$,
we conclude that $\mu_{F(A)}$ is an isomorphism. Hence $\mu_B$
is an isomorphism and this proves the lemma.
\end{proof}

We are in a position to prove the main result of this section, which we state in the next theorem.

\begin{theorem} \label{teo. equivalencia general}
 Let $\mathcal{C}$ be a class of objects in $\modu\,(\Lambda)$, $M\in\leftidx{^{\bot_1}}{\mathcal{C}}$,
 $\Gamma=\End_\Lambda(M)^{op}$, $F=\Hom_\Lambda(M,-):\modu\,(\Lambda)\rightarrow \modu\,(\Gamma)$ and
 $G= M\otimes_\Gamma -:\modu\,(\Gamma)\rightarrow \modu\,(\Lambda)$. If $\mathcal{F}(\mathcal{C})\subseteq C^{M}_2$,
 then the following statements hold.
\begin{itemize}
\item[(a)] $F|_{\F(\C)}:\F(\C)\to \F(F(\C))$ is an exact equivalence of categories and
$G|_{\F(F(\C))}:\F(F(\C))\to \F(\C)$ is a quasi-inverse of
$F|_{\F(\C)}.$
\item[(b)] If $\add\,(M)\subseteq\F(\C)$ and $\F(F(\C))$ is closed under kernels of epimorphisms, then
$\F(F(\C))$ is resolving and $\add\,(M)=\F(\C)\cap
\leftidx{^{\bot_1}}\F(\C).$
\end{itemize}
\end{theorem}
\begin{proof}
Let $\mathcal{F}(\mathcal{C})\subseteq C^{M}_2$ and recall that
${}^{\perp_1}{\C}={}^{\perp_1}{\F(\C)}$. \

(a) By Theorem \ref{teorema de Auslander} (a), we have that
$F|_{\mathcal{F}(\mathcal{C})}:\mathcal{F}(\mathcal{C})\rightarrow
F(\mathcal{F}(\mathcal{C}))$ is an equivalence of categories.
Furthermore, since $M\in{}^{\perp_1}{\F(\C)},$ it follows that
$F|_{\mathcal{F}(\mathcal{C})}$ is exact. Then, by Corollary \ref{Corol.
F(C)=mathcal{F}(F(C))}, we get that
$F(\mathcal{F}(\mathcal{C}))=\mathcal{F}(F(C))$ and this proves the
first claim in (a). The rest of the proof of (a) follows immediately
from Proposition \ref{prop. C1M-isom.epsilon} and Lemma \ref{lema cuasi-inversa de F}.
\

(b) Let $\text{\rm add}\,(M)\subseteq\mathcal{F}(\mathcal{C})$ and
let $\mathcal{F}(F(\mathcal{C}))$ be closed under kernels of
epimorphisms. By Theorem \ref{teorema de Auslander} (c), we
have that $F|_{\text{\rm add}\,(M)}:\text{\rm add}\,(M)\rightarrow
\proj\,(\Gamma)$ is an equivalence and therefore
$\proj\,(\Gamma)\subseteq\F(F(\C))$. Then, by Lemma \ref{X resolvente},
$\proj\,(\Gamma)=\F(F(\C))\cap{}^{\bot_1}\F(F(\C))$.\\
The hypotheses imply that $\add\,(M)\subseteq\F(\C)\cap
\leftidx{^{\bot_1}}\F(\C).$  Let $A\in\mathcal{F}(\mathcal{C})\cap
\leftidx{^{\bot_1}}{\mathcal{F}(\mathcal{C})}$. Then
$F(A)\in\mathcal{F}(F(\mathcal{C}))$ and, since
$\mathcal{F}(F(\mathcal{C}))$ is resolving, there exists an exact
sequence in $\mathcal{F}(F(\mathcal{C}))$
$$\varepsilon:0\rightarrow Z'\rightarrow P\rightarrow F(A)\rightarrow 0$$
 with $P\in\proj\,(\Gamma)$. By (a), we have that $Z'\simeq F(Z)$ for some $Z\in\mathcal{F}(\mathcal{C})$. Hence,
 by Proposition \ref{lema Tor y lema2} (b), there exists an exact sequence in $\mathcal{F}(\mathcal{C})$
$$\eta:0\rightarrow Z\rightarrow Q\rightarrow A\rightarrow 0$$
such that $F(\eta)\simeq\varepsilon$. Since $A\in\leftidx{^{\bot_1}}{\mathcal{F}(\mathcal{C})}$,
then $\eta$ splits, and so does $\varepsilon$. Thus $F(A)\in\proj\,(\Gamma)=F(\text{\rm add}\,(M))$. Consequently,
$A\in\text{\rm add}\,(M)$.

\end{proof}

The above results can be applied to the study of proper costratifying systems, which will be introduced in the next section.
We will show next that they can also be used to obtain,
in a unified  way, known results about Ext-projective stratifying
systems. To do so, we start with two lemmas. The first one
states a result proven in \cite{MMS}, which is
 fundamental for our considerations, and the second one is a useful technical
 lemma. To start with, we recall firstly the definition of Ext-projective
 stratifying system.

\begin{defin}(See \cite{MMS}.) \label{definicion de epps}
Let $\Lambda$ be an artin $R$-algebra. An Ext-projective
 stratifying system $(\Theta,\underline{Q},\leq),$ of size $t$ in $\modu\,(\Lambda),$ consists
of two families of non-zero $\Lambda$-modules
$\Theta=\{\Theta(i)\}_{i=1}^t$ and $\underline{Q}=\{ Q(i)\}_{i=1}^t,$
with $Q(i)$ indecomposable for all $i$, and a linear order $\leq$ on the
set $[1,t],$ satisfying the following conditions.
\begin{itemize}
 \item[(a)] $\Hom_\Lambda(\Theta(i),\Theta(j))=0$ if $i>j.$
 \item[(b)] For each $i\in[1,t],$ there is an exact sequence
$$\varepsilon_i: 0\longrightarrow K(i)\longrightarrow Q(i)\overset{\beta_i}{\longrightarrow}\Theta(i)\longrightarrow0,$$
with $K(i)\in \F(\{\Theta(j):j>i\}).$
 \item[(c)] $\underline{Q}\subseteq{}^{\perp_1}\Theta,$ that is, $\Ext_\Lambda
 ^{1}(Q(i),-)|_{\Theta}=0$ for any $i\in[1,n].$
\end{itemize}
\end{defin}

\begin{lemma} \label{lema MMS sist. estratif theta}
 Let $(\Theta,\underline{Q},\leq)$ be an Ext-projective stratifying system in $\modu\,(\Lambda)$ of size $t$. Then, for each
 $M\in\F(\{\Theta(j):j\geq i\})$, there exists an exact sequence
 in $\F(\Theta)$
$$0\rightarrow N\rightarrow Q_0(M)\rightarrow M\rightarrow0$$
such that $Q_0(M)\in\add\,(\bigoplus_{j\geq i}Q(j))$ and
$N\in\mathcal{F}(\{\Theta(j):j>i\})$. Moreover, for $Q=\oplus^{t}_{i=1}\,Q(i)$, $\F(\Theta)\subseteq C^Q_m$ for all $m\geq 1$.
\end{lemma}
\begin{proof} See in \cite[Proposition 2.10]{MMS}. The last statement of the lemma is proved using Definition \ref{definicion de epps} (c).
\end{proof}

For the following lemma, we consider a set $\{M_1,\cdots,M_n\}$ of
indecomposable $\Lambda$-modules which are pairwise not isomorphic,
$M=\oplus^{t}_{i=1}\,M_i,$ $\Gamma=\End_\Lambda({M})^{op}$ and
$F=\Hom_\Lambda(M,-):\modu\,(\Lambda)\rightarrow \modu\,(\Gamma).$

\begin{lemma}\label{lema ImFalpha contenida en la Traza}
Let
$M'\overset{\alpha}{\longrightarrow}M_i\overset{\beta}{\longrightarrow}X\rightarrow
0$ be an exact sequence in $\modu\,(\Lambda),$ with
$J\subseteq[1,t],$ $M'\in\add\,(\oplus_{j\in J}M_j)$ and $\beta\neq
0,$ such that the induced sequence
$$F(M')\overset{F(\alpha)}{\longrightarrow}F(M_i)\overset{F(\beta)}{\longrightarrow}F(X)\rightarrow
0$$ is exact in $\modu\,(\Gamma).$ Then the following statements
hold.
\begin{itemize}
\item [(a)] $\Ima\,(F(\alpha))\subseteq\Tr_{\oplus_{j\in J}F(M_j)}\,(\rad\,F(M_i)).$
\item [(b)] If $\Hom_\Lambda(M_j,X)=0$ for all $j\in J$, then
$$\Ima\,(F(\alpha))=\Tr_{\oplus_{j\in J}F(M_j)}\,(F(M_i)).$$
\end{itemize}
\end{lemma}
\begin{proof}
The proof is a straightforward consequence of Theorem \ref{teorema de Auslander}.
%%(a) By \ref{teorema de Auslander}, it follows that
%$F(\beta)\neq 0$ and $F(M_i)$ is an indecomposable projective
%$\Gamma$-module. Thus $\Ima\,(F(\alpha))\subseteq \rad\,F(M_i);$
%proving (a) since $M'\in\add\,(\oplus_{j\in J}M_j).$
%\
%
%(b) By (a), it is enough to show the inclusion $\Tr_{\oplus_{j\in
%J}F(M_j)}\,(F(M_i))\subseteq\Ima\,(F(\alpha)).$ Let
%$\theta\in\Hom_\Gamma(F(M_j),F(M_i))$ for some $j\in J.$ Then
%$F(\beta)\theta\in\Hom_\Gamma(F(M_j),F(X)),$ which is zero by
%\ref{teorema de Auslander} (b); and so $\Ima\,(\theta)\subseteq
%\Ima\,(\alpha),$ proving the result.
\end{proof}

We are now in a position to give a different proof of the following known result (\cite{ES}, \cite{MMS}; see also \cite{Webb} for related results).

\begin{theorem}
\label{teo. sist. estratif. Ext-proyectivos MMS}
 Let $(\Theta,\underline{Q},\leq)$ be an Ext-projective stratifying
system of size
$t$ in $\modu\,(\Lambda),$
 $Q=\oplus^{t}_{i=1}\,Q(i)$, $\Gamma=\End_\Lambda({Q})^{op}$,
 $F=\Hom_\Lambda(Q,-):\modu\,(\Lambda)\rightarrow \modu\,(\Gamma)$ and
 $G= Q\otimes_\Gamma -:\modu\,(\Gamma)\rightarrow \modu\,(\Lambda)$.
 Then, the following statements hold.
\begin{itemize}
\item[(a)] The family ${}_\Gamma P=\{F(Q(i)):i\in[1,t]\}$ is a
representative set of the indecomposable projective
$\Gamma$-modules. In particular, $\Gamma$ is a basic algebra and
$\rk\,K_0(\Gamma)=t.$
\item[(b)] $(\Gamma,\leq)$ is a standardly stratified algebra, that is, $\proj\,(\Gamma)\subseteq
\F(\leftidx{_{\Gamma}}{\Delta}).$
\item[(c)] The restriction $F|_{\F(\Theta)}:\F(\Theta)\to \F(\leftidx{_{\Gamma}}{\Delta})$ is an exact equivalence
of categories and
$G|_{\F(\leftidx{_{\Gamma}}{\Delta})}:\F(\leftidx{_{\Gamma}}{\Delta})\to
\F(\Theta)$ is a quasi-inverse of $F|_{\F(\Theta)}.$
\item[(d)] $F(\Theta(i))\simeq\leftidx{_{\Gamma}}{\Delta}(i)$, for all $i\in[1,t]$.
\item[(e)] $\text{\rm add}\,(Q)=\mathcal{F}(\Theta)\cap \leftidx{^{\bot_1}}{\mathcal{F}(\Theta)}$.
\end{itemize}
\end{theorem}
\begin{proof} (a) follows from the fact that $Q=\{Q(i)\}_{i=1}^t$ is a family of
indecomposable and pairwise not isomorphic $\Lambda$-modules (see
\cite[II Proposition 2.1]{ARS}).
\

 On the other hand, by Lemma \ref{lema MMS sist. estratif theta}
we know that $\mathcal{F}(\Theta)\subseteq C^{Q}_2$, so the
hypotheses of Theorem \ref{teo. equivalencia general} are satisfied for
$\mathcal{C}=\Theta$ and $M=Q$. Furthermore, the same lemma implies, for each $i\in[1,t],$ the existence of a
presentation
$$Q'\overset{\alpha_i}{\longrightarrow} Q(i)\overset{\beta_i}{\longrightarrow} \Theta(i)\rightarrow0$$
with $Q'\in\text{\rm add}(\bigoplus_{j>i}Q(j))$ and such that the
induced sequence
$$F(Q')\overset{F(\alpha_i)}{\longrightarrow} F(Q(i))\overset{F(\beta_i)}{\longrightarrow} F(\Theta(i))\longrightarrow0$$
is exact in $\modu\,(\Gamma).$ Since
$\Hom_\Lambda(Q(j),\Theta(i))=0$ for all $j>i$ (see \cite[Lemma
2.6 (b)]{MMS}),
 we conclude from Lemma \ref{lema ImFalpha contenida en la Traza} (b) that
 $$\Ima\,(F(\alpha_i))=\Tr_{\oplus_{j>i}\,F(Q(j))}\,F(Q(i)).$$ But, according with (a),
 the standard $\Gamma$-modules are the factors $\leftidx{_{\Gamma}}{\Delta}(i)=F(Q(i))/\Tr_{\oplus_{j>i}F(Q(j))}F(Q(i)).$
 Hence $\leftidx{_{\Gamma}}{\Delta}(i)=F(Q(i))/\Ima\,(F(\alpha_i))\simeq F(\Theta(i))$ for all $i\in[1,t]$.
Items (c) and (d) follow now from Theorem \ref{teo. equivalencia general} (a).
 \

On the other hand, since $\text{\rm
add}\,(Q)\subseteq\mathcal{F}(\Theta)$ and
$F(\mathcal{F}(\Theta))=\mathcal{F}(\leftidx{_{\Gamma}}{\Delta})$ is
closed under kernels of epimorphisms (see \cite{DR}, \cite{Xi}), we
can apply Theorem \ref{teo. equivalencia general} and obtain that (b)
and (e) hold.
\end{proof}

\section{Proper costratifying systems.}
In this section we introduce the notion of a proper costratifying system \linebreak $(\Psi,\Q,\leq)$ and illustrate it with
some examples. We also show that the notions of $\Psi$-length and $\Psi$-multiplicity are well defined.

\begin{defin}\label{def de sist. estricto}
Let $\Lambda$ be an artin $R$-algebra. A proper costratifying system
$(\Psi,\Q,\leq)$, of size $t$ in $\modu\,(\Lambda)$, consists of two
families of $\Lambda$-modules $\Psi=\{\Psi(i)\}_{i=1}^t$ and $\Q=\{
Q(i)\}_{i=1}^t,$ with $Q(i)$ indecomposable for all $i$, and a
linear order $\leq$ on the set $[1,t]$, satisfying the following
conditions.
\begin{itemize}
 \item[(a)] $\End_\Lambda(\Psi(i))$ is a division ring for all $i\in[1,t].$
 \item[(b)] $\Hom_\Lambda(\Psi(i),\Psi(j))=0$ if $i<j.$
 \item[(c)] For each $i\in[1,t],$ there is an exact sequence
$$\varepsilon_i: 0\longrightarrow Z(i)\longrightarrow Q(i)\overset{\beta_i}{\longrightarrow}\Psi(i)\longrightarrow0,$$
with $Z(i)\in \F(\{\Psi(j):j\leq i\}).$
 \item[(d)] $\Q\subseteq{}^{\perp_1}\Psi$, that is, $\Ext_\Lambda
 ^{1}(Q(i),-)|_{\Psi}=0$ for any $i\in[1,n].$
\end{itemize}
We will denote by $Q$ the $\Lambda$-module
$\bigoplus_{i=1}^{t}\,Q(i).$
\end{defin}

The notion of a \emph{proper stratifying system} is defined dually.
\begin{remark}\label{lema sistemas propios isomorfos}
\rm Let $\Lambda$ be an artin $R$-algebra and $(\Psi,\Q,\leq)$ be a proper costratifying system
of size $t$ in $\modu\,(\Lambda).$ Then:
 \begin{itemize}
 \item[(a)] For any $i\in[1,t],$ the map $\beta_i:Q(i)\to \Psi(i)$
 is a right-minimal $\add\,Q$-approximation of $\Psi(i).$ Indeed, this
 follows from the fact that $Q(i)$ is indecomposable and
 $\Q\subseteq{}^{\perp_1}\Psi={}^{\perp_1}\F(\Psi).$
 \item[(b)] Let $(\Psi',\Q',\leq)$ be another proper costratifying
system of size $t$ in $\modu\,(\Lambda)$. If $\Psi(i)\simeq\Psi'(i)$ for all $i\in[1,t],$
then there is an exact and commutative diagram in $\F(\Psi)$
$$ \begin {diagram} \dgARROWLENGTH=1.5em
\node {0} \arrow {e,l}{} \node {Z(i)} \arrow {e,l}{} \arrow
{s,r}{} \node {Q(i)} \arrow {e,l}{\beta_i} \arrow
{s,r}{} \node {\Psi(i)} \arrow {e,l}{} \arrow {s,r}{}
\node {0} \\
\node {0} \arrow {e,l}{} \node {Z'(i)} \arrow {e,l}{} \node {Q'(i)}
\arrow {e,l}{\beta'_i} \node {\Psi'(i)} \arrow {e,l}{} \node {0,}
\end {diagram}$$
where the vertical arrows are isomorphisms. This statement follows from the item (a), since $\F(\Psi)=\F(\Psi').$
 \end{itemize}
\end{remark}

\begin{example}
\rm Let $(\Theta,\underline{Q},\leq)$ be an Ext-projective stratifying
system of size $t$ in
$\modu\,(\Lambda).$ If $\End_\Lambda(\Theta(i))$ is a division ring
for all $i\in[1,t]$, then
$(\Psi=\Theta,\Q=\underline{Q},\leq^{op})$ is a proper costratifying
system of size $t.$
\end{example}
\begin{example}
\rm Let $(\Psi,\Q,\leq)$ be a proper costratifying system
of size $t$ in $\modu\,(\Lambda).$ For $i\in[1,t]$, consider the families of
$\Lambda$-modules $\Psi_i=\{\Psi(j)\,:\,j\leq i\}$ and $\Q_i=\{Q(j)\,:\,j\leq i\}.$ Then, $(\Psi_i,\Q_i,\leq)$ is a proper costratifying system
in $\modu\,(\Lambda)$, with size less or equal than $t$.
\end{example}

\begin{example}\label{canonical prop ss}
\rm Let $(\Lambda,\leq)$ be a standardly stratified algebra and
$T=\bigoplus_{i=1}^{n}T(i)$ be the characteristic tilting module. We
consider $\Q=\{T(1),\cdots,T(n)\}$ and
$\Psi={}_{\Lambda}\!\overline{\nabla}$ the proper costandard
modules. Then, by \cite[Lemma 1.2 (iii), Theorem 2.1 and Lemma 2.5 (iii)]{AHLU}, it follows that $(\Psi,\Q,\leq)$
is a proper costratifying system of size $n$ in $\modu\,(\Lambda)$. We say that
$({}_{\Lambda}\!\overline{\nabla},\{T(i)\}_{i=1}^n,\leq)$ is the
canonical proper costratifying system associated to the standardly stratified algebra
$(\Lambda,\leq).$
\end{example}

\begin{example}
\label{ejemplo 0, Q=T}
\rm The following is an example of a proper costratifying system $(\Psi,\Q,\leq)$ such that
$\Psi\neq{}_\Lambda\!\overline{\nabla}$ and $(\Lambda,\leq)$ is a standardly stratified algebra.
 Let $\Lambda$ be the path algebra of the quiver
$$\underset{1}{\circ}\longrightarrow\underset{2}{\circ}\longrightarrow\underset{3}{\circ}.$$
Consider the natural order on $\{1,2,3\}$. The proper
costandard $\Lambda$-modules can be described as follows
$${}_\Lambda\!\overline{\nabla}(1)=1, \quad {}_\Lambda\!\overline{\nabla}(2)={1 \atop 2},
\quad {}_\Lambda\!\overline{\nabla}(3)=\begin{array}{c}1 \\ 2 \\
3\end{array} .$$ Now, consider
$\Psi=\{\Psi(1)=3,\;\Psi(2)=1,\;\Psi(3)={1 \atop 2}\}$ and $$\Q=\{Q(1)=3,\;
Q(2)=1,\; Q(3)=\small\begin{array}{c}1 \\2 \\
3\end{array}\}.$$ Then $(\Psi,\Q,\leq)$ is a proper costratifying system of size $3$ in $\modu\,(\Lambda),$ which is not the canonical one.
\end{example}

\begin{example}
\label{ejemplo 1, Q distinto T} \rm The following is an example of a
proper costratifying system $(\Psi,\Q,\leq)$ such that
$\Psi\neq{}_\Lambda\!\overline{\nabla}$ and $(\Lambda,\leq)$ is not a
standardly stratified algebra.
\

 Let $\Lambda$ be given by the quiver
$$\xymatrix{
 {\circ}
       & {\circ} \ar[l]^\alpha \ar@(ur,ul)[]_{\beta}
              & {\circ} \ar[l]^\gamma }$$\vspace{-0.7cm}
$$\xymatrix{1&2&3}$$
 with the relations $\beta^{2}=0$, $\beta\alpha=0$ and $\gamma\beta=0$.
Consider the natural order $\leq$ on $\{1,2,3\}$, and the sets
$$\Psi=\{\Psi(1)=2,\;\Psi(2)=\small\begin{array}{c}3 \\2
\\1\end{array},\;\Psi(3)={2 \atop 1}\}$$ and
$$\Q=\{Q(1)=\small\begin{array}{c}2\\2\end{array},\; Q(2)=\small\begin{array}{c}3
\\2 \\1\end{array},\; Q(3)=\small\begin{array}{ccc} & 2 &  \\1 &  & 2\end{array}\}.$$ Then
$(\Psi,\Q,\leq)$ is a proper costratifying system of size $3$ in
$\modu\,(\Lambda),$ and $\Psi\neq{}_\Lambda\!\overline{\nabla}.$
\end{example}

\begin{lemma} \label{prop. basicas de sistema estricto}
Let $(\Psi,\Q,\leq)$ be a proper costratifying system of size $t$ in
$\modu\,(\Lambda).$ If $i<j$ then
 $$\Hom_\Lambda(Q(i),\Psi(j))=0=\Hom_\Lambda(Z(i),\Psi(j))\quad\text{ and }\quad\Ext_{\Lambda}^{1}(\Psi(i),\Psi(j))=0.$$
\end{lemma}

\begin{proof} Let $i<j.$ By Definition \ref{def de sist. estricto} $(c)$, there is an exact
sequence in $\F(\Psi)$
$$\varepsilon_i\;:\quad 0\longrightarrow Z(i)\longrightarrow Q(i)\longrightarrow\Psi(i)\longrightarrow0.$$
Applying the functor $\Hom_\Lambda(-,\Psi(j))$ to $\varepsilon_i,$
we get the exact sequence

\bigskip

$0\longrightarrow \Hom_\Lambda(\Psi(i),\Psi(j))\longrightarrow
\Hom_\Lambda(Q(i),\Psi(j))\longrightarrow
\Hom_\Lambda(Z(i),\Psi(j))\longrightarrow $

\medskip

$\Ext_\Lambda ^{1}(\Psi(i),\Psi(j))\longrightarrow 0.$

\bigskip

 We know that $Z(i)\in
\F(\{\Psi(\lambda):\lambda\leq i\})$ and, since $\lambda\leq i<j$,\\ $\Hom_\Lambda(\Psi(\lambda),\Psi(j))=0$ (see Definition \ref{def de sist.
estricto} (b), (c)). Then, it is easy to see that $\Hom_\Lambda(Z(i),\Psi(j))=0$. Finally, the lemma follows from the last sequence.
\end{proof}

K. Erdmann and C. Saenz proved in \cite{ES} that the filtration
multiplicity $[M:\Theta(i)]$ of $\Theta(i)$ in a $\Theta$-filtered
$\Lambda$-module $M$ is well defined, for the relative simple module
$\Theta(i)$ associated to a stratifying system $(\Theta,\leq).$ The
same result holds for the relative simple module $\Psi(i)$ of a
proper costratifying system, as we state in the following lemma.

\begin{lemma}\label{lema multiplic. filtracion; Yi no iso a Yj}
Let $(\Psi,\Q,\leq)$ be a proper costratifying system of size $t$ in
$\modu\,(\Lambda)$. Then the following statements hold.
\begin{itemize}
 \item[(a)] For any $M\in\F(\Psi)$, the filtration multiplicity $[M:\Psi(i)]_\xi$ of $\Psi(i)$ in $M$
 does not depend on the given $\Psi$-filtration $\xi$ of $M$.
 \item[(b)] $Q(i)\not\simeq Q(j) $ if $i\neq j$.
\end{itemize}
\end{lemma}

\begin{proof}
(a) The proof is dual to the one given in \cite[Lemma 1.4]{ES} (see
also \cite[Lemma 2.6]{MMS}), which can be adapted by using
Lemma \ref{prop. basicas de sistema estricto} and length
instead of dimension.
 \

(b) Suppose that $Q(i)\simeq Q(j)$ and $i<j$. By (a) and Definition \ref{def de
sist. estricto}, we know that $[Q(i):\Psi(j)]=0$ and
$[Q(j):\Psi(j)]>0$, contradicting our first assumption.
\end{proof}

Given a proper costratifying system $(\Psi,\Q,\leq)$ of size $t$
in $\modu\,(\Lambda)$, the above lemma shows that the filtration multiplicity is well
defined. Thus we can define the function \emph{$\Psi$-length}
$\ell_{\Psi}:\F(\Psi)\to\N$ as follows,
$\ell_{\Psi}(M)=\sum_{i=1}^{t}[M:\Psi(i)].$ It can be seen that the
$\Psi$-length is an additive function, that is, for any exact sequence
$0\to N\to E\to M\to 0$ in $\F(\Psi),$ we have that
$\ell_{\Psi}(E)=\ell_{\Psi}(N)+\ell_{\Psi}(M).$

\begin{lemma}\label{reordenamiento de la tita-filtracion}
Let $\Psi=\{\Psi(i)\}^{t}_{i=1}$ be a family of $\Lambda$-modules
satisfying that\\ $\Ext_\Lambda^{1}(\Psi(i),\Psi(j))=0$ for $i<j$.
Then, for all $M\in \F(\Psi)$, any $\Psi$-filtration of
$M$ can be rearranged, with the same $\Psi$-composition factors, to
get a $\Psi$-filtration $0=M_0\subseteq M_1\subseteq\cdots\subseteq
M_s=M$ such that $M_i/M_{i-1}\simeq\Psi(k_i)$, with
$k_1\leq\cdots\leq k_s$.
\end{lemma}
\begin{proof}
The proof is based on the following observation. Let $Z\subseteq
Y\subseteq X$ be a chain of $\Lambda$-submodules such that
$X/Y\simeq A$ and $Y/Z\simeq B$. If $\Ext_\Lambda^{1}(A,B)=0$ then
there exists a $\Lambda$-submodule $W$ such that $Z\subseteq
W\subseteq X$ with $X/W\simeq B$ and $W/Z\simeq A$.
\end{proof}

The following result is the straightforward generalization of
\cite[Lemma 1.7]{AHLU} to the context of proper costratifying systems
$(\Psi,\Q,\leq).$ This lemma shows, in particular, that the
$\Psi(i)'s$ behave in some sense as simple objects in $\F(\Psi)$,
since non-zero morphisms into them are surjective, and it is
fundamental in all that follows.

\begin{lemma}\label{lema sobre f suryectiva y el Kerf}
Let $(\Psi,\Q,\leq)$ be a proper costratifying system of size $t$ in
$\modu\,(\Lambda)$, $X\in \F(\{\Psi(j):j\leq i\})$ and
$f\in\Hom_\Lambda(X,\Psi(i))$. If $f\neq 0$ then $f$ is surjective
and $\Ker\,(f)\in \F(\{\Psi(j):j\leq i\}).$
\end{lemma}
\begin{proof}
The proof in \cite{AHLU} can be adapted directly, by using that the
$\Psi$-filtration multiplicity is well defined (see Lemma \ref{lema
multiplic. filtracion; Yi no iso a Yj}), and Lemmas \ref{prop. basicas de
sistema estricto} and \ref{reordenamiento de la tita-filtracion}.
\end{proof}

\begin{corollary}\label{Corol. f right minimal approximation}
Let $(\Psi,\Q,\leq)$ be a proper costratifying system of size $t$ in
$\modu\,(\Lambda)$. Then any non-zero map $f\in\Hom_\Lambda(Q(i),\Psi(i))$ is a right minimal
$\add\,(Q)$-approximation of $\Psi(i)$.
\end{corollary}
 \begin{proof} Let $0\neq f\in\Hom_\Lambda(Q(i),\Psi(i))$. Then, by Lemma \ref{lema sobre f suryectiva y el Kerf},
 we have that $0\rightarrow \Ker\,(f)\rightarrow Q(i)\overset{f}{\rightarrow}\Psi(i)\rightarrow0$ is an
 exact sequence in $\F(\{\Psi(j)\,:\,j\leq i\}).$ Furthermore, since $\Ext_\Lambda^{1}(Q,\Ker\,(f))=0$ and $Q(i)$
 is indecomposable, it follows that $f$ is a right minimal $\add\,(Q)$-approximation of $\Psi(i).$
 \end{proof}

\section{The standardly stratified algebra associated to a proper costratifying system.}

In this section we prove, for a proper costratifying system $(\Psi,\Q,\leq)$, that the pair $(\text{End}_\Lambda(Q),\leq^{op})$ is a standardly stratified algebra. Moreover, the category of  modules filtered by $\Psi$ is dual to the category of modules filtered by the proper costandard modules over $\text{End}_\Lambda(Q)$. Finally, we show that $\mathcal{F}(\Psi)$ is \linebreak coresolving precisely when $\Psi$ coincides with the costandard modules of a standardly stratified algebra.\\
The following proposition is important for our considerations, because it will allow us to apply the results in Section 2.
\begin{proposition}\label{dim(Hom(Yi,titai))=1; sucesion exacta M', E y M}
Let $(\Psi,\Q,\leq)$ be a proper costratifying system of size $t$ in
$\modu\,(\Lambda).$ Then, for each $M\in \F(\{\Psi(j)\,:\,j\leq
i\}),$ there exists an exact sequence $0\longrightarrow
M'\longrightarrow Q'\longrightarrow M\longrightarrow0$ in
$\F(\{\Psi(j)\,:\,j\leq i\})$ with $Q'\in \add\,(\bigoplus_{j\leq
i}Q(j))$. In particular $\F(\Psi)\subseteq C^Q_m$ for all $m\geq 1$.
\end{proposition}

\begin{proof} Let $M\in \F(\{\Psi(j)\,:\,j\leq
i\}).$ We proceed by induction on $\ell_{\Psi}(M).$ For
$\ell_{\Psi}(M)=1$, we get the sequence from Definition \ref{def de sist.
estricto} (c).
\

 Let $\ell_{\Psi}(M)>1$. Then, there is an exact
sequence in $\F(\{\Psi(j)\,:\,j\leq i\})$
$$0\longrightarrow\Psi(i_1)\stackrel{\alpha}{\longrightarrow} M\stackrel{\gamma}{\longrightarrow} M_1\rightarrow 0,$$
with $\ell_{\Psi}(M_1)<\ell_{\Psi}(M).$ Hence, by induction, there
exists an exact sequence in $\F(\{\Psi(j)\,:\,j\leq i\})$
$$0\longrightarrow M'_1\longrightarrow Q'_1\stackrel{\beta}{\longrightarrow} M_1\longrightarrow 0$$
with $Q'_1\in\add\,(\bigoplus_{j\leq i}Q(j)).$ Since $Q\in
\leftidx{^{\bot_1}}{\F(\Psi)}$, there is a morphism
$\overline{\beta}:Q'_1\to M$ such that
$\beta=\gamma\overline{\beta}$. Hence we get an exact and commutative diagram
$$\begin{CD}@. 0 @. 0 @. 0\\
@.@VVV @VVV @VVV\\
0@>>> Z(i_1) @>>> X_2 @>>> M'_1 @>>> 0 \\
 @. @VVV  @VVV @VVV \\
0@>>>Q(i_1) @>{\left( 1 \atop 0\right)}>> Q(i_1)\oplus Q'_1 @>{(0,1)}>> Q'_1 @>>> 0 \\
@. @V{\beta_{i_1}}VV @V{f}VV @VV{\beta}V\\
0@>>> \Psi(i_1) @>>{\alpha}> M @>>{\gamma}> M_1 @>>> 0 \\
@.@VVV @VVV @VVV\\
@. 0 @. 0 @. 0
\end{CD}  $$
where $f=(\alpha\beta_{i_1},\, \overline{\beta})$ and
$X_2=\Ker\,(f)$. Then $X_2\in\F(\{\Psi(j)\,:\,j\leq i\})$ and the
middle vertical sequence is the desired one.
\end{proof}

\begin{corollary}\label{Corol. solo equiv. para sist. propios}
Let $(\Psi,\Q,\leq)$ be a proper costratifying system of size $t$ in
$\modu\,(\Lambda),$ $\Gamma=\End_\Lambda(Q)^{op},$
$F=\Hom_\Lambda(Q,-):\modu\,(\Lambda) \rightarrow\modu\,(\Gamma)$
and $G=Q\otimes_{\Gamma}-:\modu\,(\Gamma)
\rightarrow\modu\,(\Lambda)$. Then, the following statements hold.
 \begin{itemize}
 \item[(a)] The restriction $F|_{\F(\Psi)}:\F(\Psi)\rightarrow\F(F(\Psi))$ is an
exact equivalence of categories and
$G|_{\F(F(\Psi))}:\F(F(\Psi))\rightarrow\F(\Psi)$ is a quasi-inverse
of $F|_{\F(\Psi)}.$
 \item[(b)] If $\F(F(\Psi))$ is closed under kernels of
 epimorphisms, then $$\text{\rm add}\,(Q)=\mathcal{F}(\Psi)\cap
 \leftidx{^{\bot_1}}{\mathcal{F}(\Psi)}.$$
 \end{itemize}
\end{corollary}

\begin{proof} By Proposition \ref{dim(Hom(Yi,titai))=1; sucesion exacta M', E y M}, we know that
$\mathcal{F}(\Psi)\subseteq C^Q_2.$ On the other hand, since $Q(i)$ is indecomposable for each
$i$ and $\F(\Psi)$ is closed under extensions, it follows that
$\add\,(Q)\subseteq\F(\Psi).$ Hence, the hypotheses of Theorem \ref{teo.
equivalencia general} are satisfied for $\mathcal{C}=\Psi$ and
$M=Q$, and so the result follows.
\end{proof}

We will prove that the family $\{F(\Psi(i))\}_{i=1}^t$ coincides with the family of proper
standard modules over $\Gamma$. This fact and the previous
corollary will lead us to the main result of this section, which we
state in the following theorem.

\begin{theorem}\label{Teo. equivalencia}
Let $(\Psi,\Q,\leq)$ be a proper costratifying system of size $t$ in
$\modu\,(\Lambda)$, $\Gamma=\End_\Lambda(Q)^{op}$,
$F=\Hom_\Lambda(Q,-):\modu\,(\Lambda) \rightarrow \modu\,(\Gamma)$
and $G=Q\otimes_\Gamma
-:\modu\,(\Gamma)\rightarrow\modu\,(\Lambda)$. Let $\leftidx{_{\Gamma}}{\overline{\Delta}}=\{\leftidx{_{\Gamma}}{\overline{\Delta}}(i)\,:\,
i\in [1,t]\}$ be the proper standard modules corresponding to the pair $({}_\Gamma P,\leq^{op}),$ where $\leq^{op}$ is the opposite order of
$\leq$. Then, the following
statements hold.
\begin{itemize}
\item[(a)] The family ${}_\Gamma P=\{F(Q(i))\,:\,i\in[1,t]\}$ is a
representative set of the indecomposable projective
$\Gamma$-modules. In particular $\Gamma$ is a basic algebra and
$\rk\,K_0(\Gamma)=t.$
\item[(b)] The restriction $F|_{\F(\Psi)}:\F(\Psi)\to
\F(\leftidx{_{\Gamma}}{\overline{\Delta}})$ is an exact equivalence
of categories and
$G|_{\F(\leftidx{_{\Gamma}}{\overline{\Delta}})}:\F(\leftidx{_{\Gamma}}{\overline{\Delta}})\to
\F(\Psi)$ is a quasi-inverse of $F|_{\F(\Psi)}.$
\item[(c)] $F(\Psi(i))\simeq\leftidx{_{\Gamma}}{\overline{\Delta}}(i)$, for all
$i\in[1,t]$.
\item[(d)] $(\Gamma^{op},\leq^{op})$ is a standardly stratified
algebra.
\item[(e)] $\text{\rm add}\,(Q)=\mathcal{F}(\Psi)\cap \leftidx{^{\bot_1}}{\mathcal{F}(\Psi)}$.
\item[(f)] $\F(\leftidx{_{\Gamma}}{\overline{\Delta}})$ is resolving
and closed under direct summands in $\modu\,(\Gamma).$
\end{itemize}
\end{theorem}

\begin{proof} It is well known that the functor $F:\modu\,(\Lambda) \rightarrow
\modu\,(\Gamma)$ induces, by restriction, an equivalence from
$\add\,(Q)$ to  $\proj\,(\Gamma)$ (see \cite[II Proposition 2.1]{ARS}).
\

(a) Since $Q=\{Q(i)\}_{i=1}^t$ is a family of indecomposable and
pairwise not isomorphic $\Lambda$-modules (see Lemma \ref{lema multiplic.
filtracion; Yi no iso a Yj} (b)), we get (a) from the above observation.
\

(b) and (c): From Corollary \ref{Corol. solo equiv. para sist. propios}, we
know that the restriction
$F|_{\F(\Psi)}:\F(\Psi)\rightarrow\F(F(\Psi))$ is an exact
equivalence of categories and
$G|_{\F(F(\Psi))}:\F(F(\Psi))\rightarrow\F(\Psi)$ is a quasi-inverse
of $F|_{\F(\Psi)}.$ So, to get (b) and (c), it is enough to prove
that $$F(\Psi(i))\simeq {}_{\Gamma}\overline{\Delta}(i)={}_\Gamma
 P(i)/\Tr_{\bigoplus_{j\geq^{op}i}\,{}_\Gamma P(j)}\,(\rad\,{}_\Gamma
 P(i)),\,
\text{\rm  for all } i\in [1,t].$$ Let $i\in[1,t]$. Then, from Definition \ref{def de sist.
estricto} (c), we have an exact sequence
$$0\rightarrow Z(i)\overset{\alpha_i}{\longrightarrow} Q(i)\overset{\beta_i}{\longrightarrow}\Psi(i)\rightarrow 0,$$
where $Z(i)\in\F(\{\Psi(j)\,:\,j\leq i\}).$ Hence, by Proposition
\ref{dim(Hom(Yi,titai))=1; sucesion exacta M', E y M}, we get an
exact sequence
$$0\longrightarrow\Ker\,(t)\longrightarrow Q'\stackrel{t}{\longrightarrow}Z(i)\longrightarrow
0$$ in $\F(\{\Psi(j)\,:\,j\leq i\}),$ where
$Q'\in\add\,(\bigoplus_{j\leq i}Q(j)).$ Therefore, since $F$ is
exact on $\mathcal{F}(\Psi),$ we have a presentation
$Q'\overset{\alpha_i\,t}{\longrightarrow}
Q(i)\overset{\beta_i}{\longrightarrow}\Psi(i)\rightarrow 0$
such that $F(Q')\overset{F(\alpha_i\,t)}{\longrightarrow}
F(Q(i))\overset{F(\beta_i)}{\longrightarrow}F(\Psi(i))\rightarrow 0$
is exact in $\modu\,(\Gamma).$ So, applying Lemma \ref{lema ImFalpha
contenida en la Traza} (a), it follows that
$$\Ima\,(F(\alpha_i))=\Ima\,(F(\alpha_i\,t))\subseteq\Tr_{\bigoplus_{j\leq
i}F(Q(j))}\,(\rad\,F(Q(i))).$$ So, in order to prove that
$F(\Psi(i))\simeq\leftidx{_{\Gamma}}{\overline{\Delta}}(i),$ it is
enough to show the inclusion $\Tr_{\bigoplus_{j\leq
i}F(Q(j))}\,(\rad\,F(Q(i)))\subseteq\Ima\,(F(\alpha_i)).$ To prove
such inclusion, we assume that $j\leq i$ and consider a morphism
$\delta:F(Q(j))\rightarrow\rad\,F(Q(i))$. Let
$\imath:\rad\,F(Q(i))\rightarrow F(Q(i))$ be the inclusion map,
which is not an isomorphism since $F(Q(i))\in\proj\,(\Gamma)$ is
indecomposable (see (a)). Furthermore, from the equivalence
$F|_{\add\,(Q)}:\add\,(Q)\rightarrow\proj\,(\Gamma),$ there is a
morphism $\eta:Q(j)\rightarrow Q(i)$ such that
$\imath\delta=F(\eta)$. Hence
$\Ima\,(\delta)\subseteq\Ima\,(F(\eta)).$ We assert that
$\Ima\,(F(\eta))\subseteq \Ima\,(F(\alpha_i))$ and, from this, it follows that
$\Ima\,(\delta)\subseteq\Ima\,(F(\alpha_i)),$ proving that
$\Tr_{\bigoplus_{j\leq
i}F(Q(j))}\,(\rad\,F(Q(i)))\subseteq\Ima\,(F(\alpha_i)).$ So, to
prove that $\Ima\,(F(\eta))\subseteq \Ima\,(F(\alpha_i)),$ we need
to show that  $F(\beta_i)F(\eta)=0$ since we have the following
exact sequence
$$0\longrightarrow\Ima\,(F(\alpha_i))\longrightarrow F(Q(i))\stackrel{F(\beta_i)}{\longrightarrow}F(\Psi(i))\longrightarrow
0.$$ Thus, we only need to prove that the composition
$Q(j)\stackrel{\eta}{\rightarrow}Q(i)\stackrel{\beta_i}{\rightarrow}\Psi(i)$
is zero. If $j<i$ this is true since, by Lemma \ref{prop. basicas de
sistema estricto}, we know that $\Hom_{\Lambda}(Q(j),\Psi(i))=0.$\\
Let $i=j$ and suppose that $\beta_i\eta\neq 0$. By Corollary \ref{Corol. f
right minimal approximation}, we know that
$\beta_i\eta:Q(i)\rightarrow \Psi(i)$ and $\beta_i:Q(i)\rightarrow
\Psi(i)$ are both minimal right $\add\,(Q)$-approximations of
$\Psi(i).$ Thus, from the commutative diagram
$$ \begin {diagram} \dgARROWLENGTH=1.5em
\node {Q(i)} \arrow {e,l}{\beta_i\eta} \arrow {s,r}{\eta}{}
\node {\Psi(i)} \arrow {s,r,=}{}\\
\node {Q(i)} \arrow {e,r}{\beta_i} \node {\Psi(i)}
\end{diagram}$$
we get that $\eta$ is an isomorphism. Therefore
$F(\eta)=\imath\delta$ is also an isomorphism, contradicting that
the inclusion map $\imath:\rad\,F(Q(i))\rightarrow F(Q(i))$ is not
an isomorphism, and therefore $\beta_i\eta=0$ as desired.
\

(d) The fact that $(\Gamma^{op},\leq^{op})$ is a standardly stratified algebra is equivalent to the condition
 ${}_\Gamma\Gamma\in\F(\leftidx{_{\Gamma}}{\overline{\Delta}})$ (see \cite[2.2]{Dlab},
\cite[2.2]{ADL} or \cite{L} ). It is easy to check the last claim. In fact,
$Q\in \mathcal{F}(\Psi)$ and so $\Gamma_{\Gamma}\simeq F(Q)\in
\mathcal{F}(\leftidx{_{\Gamma}}{\overline{\Delta}}).$
\

(e) and (f): Since $(\Gamma^{op},\leq^{op})$ is a standardly
stratified algebra (see (d)), it follows from \cite[Theorem 1.6
(ii)]{AHLU} that $\F({}_{\Gamma^{op}}\!\overline{\nabla})$ is
coresolving. We get by duality
that $\F(\leftidx{_{\Gamma}}{\overline{\Delta}})$ is resolving.
On the other hand, from (b), we know that
$\F(F(\Psi))=\F(\leftidx{_{\Gamma}}{\overline{\Delta}}).$ Hence, (e) follows from Corollary \ref{Corol. solo equiv. para sist. propios}.
Finally, we prove that
$\F(\leftidx{_{\Gamma}}{\overline{\Delta}})$ is closed under direct
summands in $\modu\,(\Gamma).$ Indeed, we have that
$D_{\Gamma}(\mathcal{F}(\leftidx{_{\Gamma}}{\overline{\Delta}}))=
\mathcal{F}(\leftidx{_{{\Gamma}^{op}}}{\overline{\nabla}})=
\mathcal{F}(\leftidx{_{{\Gamma}^{op}}}{\Delta})^{\bot_1}$ (the
last equality follows from \cite[Theorem 1.6 (iv)]{AHLU}), and so
the result follows observing that
$\mathcal{F}(\leftidx{_{{\Gamma}^{op}}}{\Delta})^{\bot_1}$ is closed
under direct summands in $\modu\,(\Gamma^{op})$.
\end{proof}

\begin{remark}
\rm We recall that an algebra $\Lambda$ is properly stratified if and only if $\leftidx{_{\Lambda}}{\Lambda}\in\F({}_{\Lambda}{\Delta})\cap\F({}_{\Lambda}{\overline{\Delta}})$ (see \cite{Dlab2}). In this case, the standard modules provide a stratifying system, and the proper standard modules a proper stratifying system.
\end{remark}

\begin{example}
\rm Let $({}_{\Lambda}\!\overline{\nabla},\{T(i)\}_{i=1}^n,\leq)$ be the
canonical proper costratifying system associated to the standardly stratified algebra
$(\Lambda,\leq)$ (see Example \ref{canonical prop ss}). Then, by Theorem \ref{Teo. equivalencia} (d), $\Gamma^{op}=\End_\Lambda(T)$ is the \lq Ringel dual\rq \,  of $\Lambda$.
\end{example}

\begin{example}\label{Observ. continua ejemplo 0}
\rm Let $(\Psi,\Q,\leq)$ be the proper costratifying system considered
in Example \ref{ejemplo 0, Q=T}. In this case, the algebra
$\Gamma^{op}=\End_\Lambda(Q)$ is given by the quiver
$$\underset{1}{\circ}\overset{\varepsilon}{\longrightarrow}\underset{3}{\circ}\overset{\mu}{\longrightarrow}\underset{2}{\circ}.$$
with the relation $\mu\varepsilon=0$. Then
$$\leftidx{_{\Gamma^{op}}}{\Gamma^{op}}=\small\begin{array}{c}
                                          1 \\
                                          3
                                        \end{array}\oplus 2\oplus\begin{array}{c}
                                                                   3 \\
                                                                   2
                                                                 \end{array}
$$
We consider $(\Gamma^{op},\leq^{op})$, where
$3\leq^{op}2\leq^{op}1.$ Then the corresponding standard modules are
$\leftidx{_{\Gamma^{op}}}{\Delta}=\{\leftidx{_{\Gamma^{op}}}{\Delta}(1)=\small\begin{array}{c}1 \\
 3\end{array}, \leftidx{_{\Gamma^{op}}}{\Delta}(2)=2, \leftidx{_{\Gamma^{op}}}{\Delta}(3)=3\}.$ In this case,
 it is easy to check directly that ${}_{\Gamma^{op}}\Gamma^{op}\in\F(\leftidx{_{\Gamma^{op}}}{\Delta})$. That is,
 $(\Gamma^{op},\leq^{op})$ is a standardly stratified algebra.
\end{example}

\begin{proposition}\label{proposicion Ext isomorfismo}
Let $(\Psi,\Q,\leq)$ be a proper costratifying system of size $t$ in
$\modu\,(\Lambda)$, $\Gamma=\End_\Lambda(Q)^{op}$ and
$F=\Hom_\Lambda(Q,-):\modu\,(\Lambda) \rightarrow \modu\,(\Gamma)$.
If $X,Y\in \mathcal{F}(\Psi)$ then the map $\rho_{X,Y}:\text{\rm
Ext}_\Lambda^{1}(X,Y)\rightarrow \text{\rm
Ext}_\Gamma^{1}(F(X),F(Y)),$ induced by $F,$ is an isomorphism.
\end{proposition}

\begin{proof}
The result is a direct consequence of Proposition \ref{prop. iso entre Ext} applied to $\mathcal{X}=\mathcal{F}(\Psi)$ and
 $M=Q,$ since Proposition \ref{dim(Hom(Yi,titai))=1; sucesion exacta M', E y M} shows
 that $\mathcal{F}(\Psi)\subseteq C^Q_m$ for any $m\geq 1.$
\end{proof}

Let $\C$ be a class of $\Lambda$-modules such that
$\add\,(Q)=\F(\C)\cap{}^{\perp_1}\F(\C)$ for some $\Lambda$-module
$Q.$ Let $M\in\mathcal{F}(\C).$ We recall that a
$\C$-\emph{projective cover} of $M,$ is a surjective morphism
$f:Q_M\rightarrow M$ of $\Lambda$-modules such that
$Q_M\in\add\,(Q),$ $\Ker\,(f)\in\F(\C)$ and $f$ is a right minimal
$\add\,(Q)$-approximation of $M.$

\begin{proposition} Let $(\Psi,\Q,\leq)$ be a proper costratifying system of size $t$ in $\modu\,(\Lambda).$
Then $\F(\Psi)$ is closed under direct summands in
$\modu\,(\Lambda),$ and any object in $\F(\Psi)$ admits a
 $\Psi$-projective cover.
\end{proposition}
\begin{proof} Recall that we have the exact equivalence
$F=\Hom_\Lambda(Q,-):\F(\Psi)\rightarrow
\F(\leftidx{_{\Gamma}}{\overline{\Delta}}),$ and also that
$\F(\leftidx{_{\Gamma}}{\overline{\Delta}})$ is closed under direct
summands in $\modu\,(\Gamma)$ (see Theorem \ref{Teo. equivalencia}).
We will carry this property back to $\mathcal{F}(\Psi).$ In fact, let $G:
\mathcal{F}(\leftidx{_{\Gamma}}{\overline{\Delta}})\rightarrow\mathcal{F}(\Psi)$
be a quasi-inverse of $F$ and $M\in \mathcal{F}(\Psi)$, and let
$M=\bigoplus_{i=1}^{n}M_i$ and $F(M)=\bigoplus_{j=1}^{m}X_j$ with $M_i$ and $X_j$ indecomposable modules for all $i,j$.
Since $\mathcal{F}(\leftidx{_{\Gamma}}{\overline{\Delta}})$ is
closed under direct summands, then $X_j$ belongs to it.
 \

We have $M\simeq GF(M)\simeq \bigoplus_{j=1}^{m}G(X_j)$. Since $G$ is
faithful and full, $G$ preserves indecomposables. Therefore, it follows
from Krull-Schmidt Theorem that $M_i\simeq G(X_{i_j})$ for some $i_j$, proving that $M_i\in \mathcal{F}(\Psi)$, as desired.
 \

We prove next that $\F(\Psi)$ admits $\Psi$-projective covers.
Indeed, by Proposition \ref{dim(Hom(Yi,titai))=1; sucesion exacta M', E y M} we
know the existence of an exact sequence in $\mathcal{F}(\Psi)$
$$0\rightarrow M'\rightarrow Q'\overset{f'}{\longrightarrow} M\rightarrow0,$$
where $f'$ is a right $\add\,(Q)$-approximation of $M$. Therefore,
since $\mathcal{F}(\Psi)$ is closed under direct summands, we get
that the right minimal version $f:Q_M\rightarrow M$ of
$f'$ is the desired $\Psi$-projective cover.
\end{proof}

The following proposition gives sufficient conditions for $F(D(\Lambda_{\Lambda}))$ to be a cotilting $\Gamma$-module.

\begin{proposition}\label{prop. add de un cotilting}
Let $(\Psi,\Q,\leq)$ be a proper costratifying system of size $t$ in
$\modu\,(\Lambda),$
$\Gamma=\End_\Lambda(Q)^{op},$
$F=\Hom_\Lambda(Q,-):\modu\,(\Lambda) \rightarrow \modu\,(\Gamma)$
and $T=F(D(\Lambda_{\Lambda}))$. Let
$\leftidx{_{\Gamma}}{\overline{\Delta}}$ be the family of proper
standard modules. If
$D(\Lambda_{\Lambda})\in \mathcal{F}(\Psi)$ and
$t=\rk\,K_0(\Lambda),$ then the following statements hold.
 \begin{itemize}
  \item[(a)] $T$ is a cotilting $\Gamma$-module.
  \item[(b)] $\mathcal{F}(\leftidx{_{\Gamma}}{\overline{\Delta}})={}^\perp T$
and $\F(\leftidx{_{\Gamma}}{\overline{\Delta}})\cap
{\mathcal{F}(\leftidx{_{\Gamma}}{\overline{\Delta}})}^{\bot_1}=\add\,(T).$
 \end{itemize}
\end{proposition}

\begin{proof} Let $D(\Lambda_{\Lambda})\in \mathcal{F}(\Psi)$ and
$t=\rk\,K_0(\Lambda).$ Since
$\Ext_\Lambda^{1}(-,D(\Lambda_\Lambda))=0,$ by Proposition
\ref{proposicion Ext isomorfismo}, it follows that $\text{\rm
Ext}_\Gamma^{1}(F(X),F(D(\Lambda_\Lambda)))=0$ for any
$X\in\F(\Psi).$ Hence
$T=F(D(\Lambda_\Lambda))\in{\mathcal{F}(\leftidx{_{\Gamma}}{\overline{\Delta}})}^{\bot_1}$
and so $\add\,(T)\subseteq
\mathcal{F}(\leftidx{_{\Gamma}}{\overline{\Delta}})\cap
{\mathcal{F}(\leftidx{_{\Gamma}}{\overline{\Delta}})}^{\bot_1}$. In
addition, from the fact that $(\Gamma^{op},\leq^{op})$ is a
standardly stratified algebra (see Theorem \ref{Teo. equivalencia}),  the duals of
\cite[Theorem 2.1, Proposition 2.2 (i)]{AHLU} show that there is a
basic cotilting $\Gamma$-module $T'$ such that
$\mathcal{F}(\leftidx{_{\Gamma}}{\overline{\Delta}})=\leftidx{^{\bot}}{T'}$
and $\mathcal{F}(\leftidx{_{\Gamma}}{\overline{\Delta}})\cap
{\mathcal{F}(\leftidx{_{\Gamma}}{\overline{\Delta}})}^{\bot_1}=\add\,(T').$
Finally, since $T'$ and $T$ have the same number of indecomposable
direct summands and $\add\,(T)\subseteq\add\,(T')$, we have
$T'\simeq T$ and the proof is complete.
\end{proof}

We know from Theorem \ref{Teo. equivalencia} (e) that the Ext-projective modules in $\mathcal{F}(\Psi)$ coincide
 with $\text{add}(Q)$. The next proposition describes the Ext-injectives in $\mathcal{F}(\Psi)$.

\begin{proposition}\label{prop.add del tilting de gama opuesta}
Let $(\Psi,\Q,\leq)$ be a proper costratifying system of size $t$ in
$\modu\,(\Lambda)$, $\Gamma=\End_\Lambda(Q)^{op}$,
$F=\Hom_\Lambda(Q,-):\modu\,(\Lambda) \rightarrow \modu\,(\Gamma)$ and $G=Q\otimes_\Gamma
-:\modu\,(\Gamma)\rightarrow\modu\,(\Lambda)$. If
${}_{\Gamma^{op}}T$ is the characteristic tilting module associated
to the standardly stratified algebra $(\Gamma^{op},\leq^{op}),$ then
$$\mathcal{F}(\Psi)\cap \mathcal{F}(\Psi)^{\bot_1}=\add\,(GD(\leftidx{_{\Gamma^{op}}}{T})).$$
\end{proposition}
\begin{proof} By Proposition \ref{proposicion Ext isomorfismo}, we know that
 $$X\in\F(\Psi)\cap \F(\Psi)^{\bot_1}\,\Leftrightarrow\,F(X)\in\F({}_{\Gamma}\overline{\Delta})
 \cap\F({}_{\Gamma}\overline{\Delta})^{\bot_1}.$$ On the other hand,
 by using \cite[Theorem 1.6 (iii), Proposition 2.2 (i)]{AHLU}, it
 follows that
 $$D(\F({}_{\Gamma}\overline{\Delta})
 \cap\F({}_{\Gamma}\overline{\Delta})^{\bot_1})=\mathcal{F}(\leftidx{_{\Gamma^{op}}}{\overline{\nabla}})\cap
 \leftidx{^{\bot_1}}{\mathcal{F}(\leftidx{_{\Gamma^{op}}}{\overline{\nabla}})}=\add\,({}_{\Gamma^{op}}T).$$
Thus, we have that $X\in\mathcal{F}(\Psi)\cap
\mathcal{F}(\Psi)^{\bot_1}$ if and only if $X\in\text{\rm
add}(GD(\leftidx{_{\Gamma^{op}}}{T})).$
\end{proof}

We recall that a class $\mathcal{X},$ of objects in
$\modu\,(\Lambda)$, is \emph{coresolving} if it is closed under
extensions, cokernels of monomorphisms and contains the injective
$\Lambda$-modules \cite{AR}. In what follows, we characterize the situation when a proper costratifying system is the canonical one.

\begin{theorem}
\label{Teo. Coresolving} Let $\Lambda$ be a basic artin algebra and $(\Psi,\Q,\leq)$ be a proper costratifying
system of size $t$ in $\modu\,(\Lambda).$ Let
$\Gamma=\End_\Lambda(Q)^{op}$ and ${}_{\Gamma^{op}}T$ be the
characteristic tilting module associated to the standardly stratified algebra
$(\Gamma^{op},\leq^{op})$. Then, the
following statements are equivalent.
\begin{itemize}
\item[(a)] $\mathcal{F}(\Psi)$ is coresolving.
\item[(b)] $\mathcal{F}(\Psi)\cap \mathcal{F}(\Psi)^{\bot_1}=\add\,(D(\Lambda_{\Lambda}))$.
\item[(c)] $D(\Lambda_{\Lambda})\in\F(\Psi)$ and $t=\rk\,K_0(\Lambda)$.
\item[(d)] $\Lambda\simeq\End({}_{\Gamma^{op}}Q)$ and ${}_{\Gamma^{op}}Q\simeq{}_{\Gamma^{op}}T.$
\item[(e)] $t=\rk\,K_0(\Lambda)$ and there is a choice of the representative set ${}_\Lambda P=\{{}_\Lambda P(i)\,:\, i
\in[1,t]\}$ of indecomposable projective $\Lambda$-modules
 such that ${}_\Lambda\!\overline{\nabla}(i)\simeq\Psi(i)$ for all $i\in[1,t]$ and $(\Lambda,\leq)$ is
 a standardly stratified algebra.
\item[(f)] $D(\Lambda_{\Lambda})\in\mathcal{F}(\Psi)$ and $Q$ is a generalized tilting $\Lambda$-module.
\end{itemize}
\end{theorem}
\begin{proof} Consider the quasi-inverse functors
$F:\F(\Psi)\to\mathcal{F}(\leftidx{_{\Gamma}}{\overline{\Delta}})$
and $G:\mathcal{F}(\leftidx{_{\Gamma}}{\overline{\Delta}})\to
\F(\Psi)$, given in Theorem \ref{Teo. equivalencia}. Then, from
\cite[Pag. 120]{CE}, we have $G=Q\bigotimes_\Gamma-\simeq
D\Hom_\Gamma(-,D(Q))$. \

The implication $(a)\Rightarrow (b)$ follows from the dual of Lemma \ref{X resolvente}.
\

$(b)\Rightarrow (d)$ Let $\mathcal{F}(\Psi)\cap
\mathcal{F}(\Psi)^{\bot_1}=\add\,(D(\Lambda_{\Lambda})).$
Then $D(\Lambda_{\Lambda})\simeq
GF(D(\Lambda_{\Lambda}))=G(\Hom_\Lambda(Q,D(\Lambda_{\Lambda})))
\simeq G(D(Q))$. In addition, by hypothesis and Proposition
\ref{prop.add del tilting de gama opuesta}, we have
$\add\,(D(\Lambda_{\Lambda}))=\add\,(GD(\leftidx{_{\Gamma^{op}}}{T}))$.
Then, since $\Lambda$ is basic, we get that $GD(\leftidx{_{\Gamma^{op}}}{T})\simeq
G(D(Q))$ and therefore
${}_{\Gamma^{op}}Q\simeq{}_{\Gamma^{op}}T.$ Now, we prove that
$\Lambda\simeq\End({}_{\Gamma^{op}}Q)$. Indeed, the isomorphisms
$D(\Lambda_{\Lambda})\simeq
G(D(Q))\simeq
D\Hom_\Gamma(D(Q),D(Q))\simeq
D\Hom_{\Gamma^{op}}(Q,Q),$ show that
$\Lambda\simeq\End({}_{\Gamma^{op}}Q).$
\

$(d)\Rightarrow (e)$ Let $\Lambda\simeq\End({}_{\Gamma^{op}}Q)$ and
${}_{\Gamma^{op}}Q\simeq{}_{\Gamma^{op}}T.$ In particular, since
${}_{\Gamma^{op}}T$ is basic, it follows that ${}_{\Gamma^{op}}Q$ is
so, and therefore $t=\rk\,K_0(\Lambda).$ On the other hand, by Example
\ref{canonical prop ss}, we know that
$(\leftidx{_{\Gamma^{op}}}{\overline{\nabla}},\{\leftidx{_{\Gamma^{op}}}{T}(i)\}_{i=1}^t,\leq^{op})$
is a proper costratifying system of size $t$ in
$\modu\,(\Gamma^{op}).$ Hence, applying Theorem \ref{Teo. equivalencia} to
this system, we get an exact equivalence
$\widetilde{F}=\Hom_{\Gamma^{op}}(T,-):\F(\leftidx{_{\Gamma^{op}}}{\overline{\nabla}})\to\F(\leftidx{_A}{\overline{\Delta}})$
such that
$\widetilde{F}(\leftidx{_{\Gamma^{op}}}{\overline{\nabla}}(i))\simeq\leftidx{_A}{\overline{\Delta}}(i)$
for all $i\in [1,t]$, with $A=\End(\leftidx{_{\Gamma^{op}}}{T})^{op}$. The same theorem implies that $(A^{op},\leq)$ is a
standardly stratified algebra and the $\leftidx{_A}{\overline{\Delta}}(i)$'s correspond to
the pair $({}_A P,\leq)$, where ${}_A P=\{{}_A P(i)=\widetilde{F}(T(i))\}_{i=1}^t$.
 \

Since we are assuming that ${}_{\Gamma^{op}}Q\simeq{}_{\Gamma^{op}}T$, we get that their endomorphism rings are isomorphic.
We will identify $\Lambda$ and $A^{op}$ through this isomorphism. Then ${}_{\Lambda}\!\overline{\nabla}=D(\leftidx{_A}{\overline{\Delta}})$,
 where the projective $A$-modules are $({}_A P(i))^*=\Hom_A({}_A P(i),A)$.
 \

 Finally, it remains to show that ${}_\Lambda\!\overline{\nabla}(i)\simeq\Psi(i)$ for all $i\in[1,t]$. Let $i\in[1,t].$ Since
$F(\Psi(i))\simeq \leftidx{_{\Gamma}}{\overline{\Delta}}(i)$, we
have \\
$\Psi(i)\simeq
GD({}_{\Gamma^{op}}\!\overline{\nabla}(i))\simeq
D\Hom_\Gamma(D({}_{\Gamma^{op}}\!\overline{\nabla}(i)),D(Q))
\simeq D\Hom_{\Gamma^{op}}(Q,
{}_{\Gamma^{op}}\!\overline{\nabla}(i))\simeq
D\Hom_{\Gamma^{op}}(T,
{}_{\Gamma^{op}}\!\overline{\nabla}(i))\simeq
D({}_A\overline{\Delta}(i))\simeq{}_\Lambda\!\overline{\nabla}(i).$
\

$(e)\Rightarrow (f)$ Assume that (e) holds. In
particular $\leftidx{_\Lambda}{\Lambda}\in
\mathcal{F}({}_\Lambda\Delta)$. Then, it follows from
\cite[2.2]{Dlab} (see also \cite{L}) that
$D(\Lambda_{\Lambda})\in\mathcal{F}({}_\Lambda\!\overline{\nabla})=\mathcal{F}(\Psi)$. If $\leftidx{_{\Lambda}}{T}=\oplus_{i=1}^t\,T(i)$ is the
characteristic tilting module associated to the standardly stratified algebra
$(\Lambda,\leq)$, we know that $({}_\Lambda\!\overline{\nabla},\{T(i)\}_{i=1}^t,\leq)$ is a proper costratifying system. From
${}_\Lambda\!\overline{\nabla}(i)\simeq\Psi(i)$, for all
$i\in[1,t]$, and the uniqueness of proper costratifying systems proven in Remark \ref{lema sistemas propios isomorfos}, it follows that ${}_\Lambda Q\simeq
{}_\Lambda T$. Hence ${}_\Lambda Q$ is a tilting module.
\

$(e)\Rightarrow (a)$ Since $(\Lambda,\leq)$ is a standardly stratified algebra, we know from
 \cite[Theorem 1.6 (ii)]{AHLU} that $\mathcal{F}({}_\Lambda\!\overline{\nabla})$ is coresolving. Furthermore,
$\F(\Psi)=\F({}_\Lambda\!\overline{\nabla})$ since ${}_\Lambda\!\overline{\nabla}(i)\simeq\Psi(i)$, para todo $i\in[1,t],$ and so (e) follows.
\

$(b)\Rightarrow (c)$ Let $\mathcal{F}(\Psi)\cap \mathcal{F}(\Psi)^{\bot_1}=\add\,(D(\Lambda_{\Lambda})).$ Then, by Proposition
 \ref{prop.add del tilting de gama opuesta}, we get that
$\add\,(D(\Lambda_{\Lambda}))=\add\,(GD(\leftidx{_{\Gamma^{op}}}{T}))$ and hence $t=\rk\,K_0(\Lambda).$
\

$(c)\Rightarrow (b)$ Let $D(\Lambda_{\Lambda})\in\F(\Psi)$ and $t=\rk\,K_0(\Lambda).$ Applying the functor $G$
to the second equality in Proposition \ref{prop. add de un cotilting} (b), we have the equalities
$\F(\Psi)\cap\F(\Psi)^{\perp_1}=$ $=G(\F({}_\Gamma\overline{\Delta})\cap\F({}_\Gamma\overline{\Delta})^{\perp_1})=
\add\,(GF(D(\Lambda_{\Lambda})))=\add\,(D(\Lambda_{\Lambda})).$
\

$(f)\Rightarrow (c)$ We have that $t=\text{\rm card}(\text{\rm ind}(\text{\rm add}\,(Q)))=\rk\,K_0(\Lambda)$, where the last
equality holds since ${}_\Lambda Q$ is tilting, and this completes our proof.
\end{proof}

\begin{remark}
\rm Let $(\Psi,\Q,\leq)$ be the proper costratifying system considered in Example \ref{ejemplo 1, Q distinto T}. In this case,
 $\Gamma^{op}=\End({}_\Lambda Q)$ is given by the quiver
$$\underset{1}{\circ}\overset{\mu}{\underset{\delta}{\rightleftarrows}}\underset{3}{\circ}\overset{\varepsilon}{\longrightarrow}
\underset{2}{\circ}$$
with the relations $\varepsilon\mu=0$ and $\mu\delta\mu=0$. By Theorem \ref{Teo. equivalencia} we know that
$(\Gamma^{op},\leq^{op})$ is a standardly stratified algebra. The characteristic tilting module is\\
 $\small\leftidx{_{\Gamma^{op}}}{T}=\small\begin{array}{c} 3 \\1 \\3 \\1\end{array}\oplus\small\begin{array}{c}3 \\2\end{array}\oplus 3$, and it
is not isomorphic to $\small\leftidx{_{\Gamma^{op}}}{Q}=\small\begin{array}{c}2 \\ 1 \end{array}\oplus\small\begin{tabular}{ccc}
                                                                                                  & 3 &   \\
                                                                                                1 &   & 2 \\
                                                                                                3 &   &   \\
                                                                                                1 &   &
                                                                                              \end{tabular}\oplus 2$.\\
On the contrary, in Example \ref{ejemplo 0, Q=T} (see also Example \ref{Observ. continua ejemplo 0}),
we have that $\small\leftidx{_{\Gamma^{op}}}{T}=\leftidx{_{\Gamma^{op}}}{Q}=\small\begin{array}{c}
                                                           3 \\
                                                           2
                                                         \end{array}\oplus 3\oplus\small\begin{array}{c}
                                                                      1 \\
                                                                      3
                                                                    \end{array}$, but
$\Lambda\not\simeq\End({}_{\Gamma^{op}}Q)$. Note that, in both
cases, we have that ${}_\Lambda\!\overline{\nabla}\neq\Psi$.
\end{remark}

{\bf Acknowledgements.} The first author was partially
supported by the Project PAPIIT-Universidad Nacional Aut\'onoma de
M\'exico IN100810-3, and by the joint research project Mexico
(CONACyT)-Argentina (MINCYT): ``Homology, stratifying systems and
representations of algebras". The second author is a researcher from CONICET, Argentina. The second and third authors acknowledge partial support from Universidad Nacional del Sur and from CONICET.

\small
\markright{}

\footnotesize

\vskip3mm \noindent Octavio Mendoza Hern\'andez\\
Instituto de Matem\'aticas\\ Universidad Nacional Aut\'onoma de M\'exico.\\
Circuito Exterior, Ciudad Universitaria\\
C.P. 04510, M\'exico, D.f., M\'exico.\\ {\tt omendoza@matem.unam.mx}

\vskip3mm \noindent Mar\'{\i}a In\'es Platzeck\\
Instituto de Matem\'atica Bah\'{\i}a Blanca,\\
Universidad Nacional del Sur.\\
8000 Bah\'{\i}a Blanca, ARGENTINA.\\
{\tt platzeck@uns.edu.ar}

\vskip3mm \noindent     Melina Vanina Verdecchia\\
Instituto de Matem\'atica Bah\'{\i}a Blanca,\\
Universidad Nacional del Sur.\\
8000 Bah\'{\i}a Blanca, ARGENTINA.\\
{\tt mverdec@uns.edu.ar}

\end{document}